\documentclass[12pt,english,reqno]{amsart}
\usepackage[T1]{fontenc}
\usepackage[latin9]{inputenc}
\usepackage{verbatim}
\usepackage{amsthm}
\usepackage{amstext}
\usepackage{amssymb}
\usepackage{cite}
\makeatletter
\numberwithin{equation}{section}
\numberwithin{figure}{section}
\theoremstyle{plain}
\newtheorem{thm}{\protect\theoremname}[section]
  \theoremstyle{definition}
  \newtheorem{defn}[thm]{\protect\definitionname}
  \theoremstyle{definition}
  \newtheorem{example}[thm]{\protect\examplename}
  \theoremstyle{plain}
  \newtheorem{lem}[thm]{\protect\lemmaname}
  \theoremstyle{plain}
  \newtheorem{prop}[thm]{\protect\propositionname}
  \theoremstyle{remark}
  \newtheorem{rem}[thm]{\protect\remarkname}
  \theoremstyle{plain}
  \newtheorem{cor}[thm]{\protect\corollaryname}

\usepackage{amscd}

\theoremstyle{plain}

\newcommand{\eqdef}{\stackrel{\mathrm{def}}{=}}

\newcommand{\N}{\mathbb{N}}
\newcommand{\R}{\mathbb{R}}
\newcommand{\C}{\mathbb{C}}
\newcommand{\Z}{\mathbb{Z}}
\newcommand{\wlim}{\operatorname{w-lim}}
\newcommand{\supp}{\mbox{supp }}

\newcommand{\vol}{{\mathrm{ vol}}}

\newcommand{\dvn}{{\mathrm dv}_{{g}^{(n)}}}

\newcommand{\dtvnx}{{\mathrm d}v_{\tilde{g}^{(n)}}(x)}

\newcommand{\dvx}{{\mathrm dv}_g(x)}
\newcommand{\dv}{{\mathrm dv}_g}

\newcommand{\nl}{\operatorname{|||}}
\newcommand{\nr}{\operatorname{|||}}

\makeatother

\usepackage{babel}
  \providecommand{\corollaryname}{Corollary}
  \providecommand{\definitionname}{Definition}
  \providecommand{\examplename}{Example}
  \providecommand{\lemmaname}{Lemma}
  \providecommand{\propositionname}{Proposition}
  \providecommand{\remarkname}{Remark}
\providecommand{\theoremname}{Theorem}

\begin{document}

\title[Defect of compactness]{Defect of compactness for Sobolev spaces on manifolds with bounded
geometry}

\author{Leszek Skrzypczak}

\thanks{One of the authors (L.S.) was supported by National Science Center,
Poland, Grant No. 2014/15/B/ST1/00164.}

\address{Faculty of Mathematics \& Computer Science, Adam Mickiewicz University,
ul. Umultowska 87, 61-614 Pozna\'{n}, Poland}

\email{lskrzyp@amu.edu.pl}

\author{Cyril Tintarev}

\thanks{The other author (C.T.) expresses his gratitude to the Faculty of Mathematics \& Computer Science of Adam Mickiewicz
University for warm hospitality.}

\address{Sankt Olofsgatan 66B, 75330 Uppsala, Sweden}

\email{tammouz@gmail.com}

\begin{abstract}
\emph{Defect of compactness}, relative to an embedding of two Banach
spaces $E\hookrightarrow F$, is the difference between a
weakly convergent sequence in $E$ and its
weak limit, taken up to a remainder that vanishes in the norm of $F$. For a number of known embeddings, Sobolev embeddings in particular, defect of compactness takes form of a \emph{profile decomposition} - a sum of clearly structured terms with asymptotically disjoint supports, called \emph{elementary concentrations}. 
In this paper we construct a profile decomposition for the Sobolev space $H^{1,2}(M)$ of a Riemannian manifold with bounded geometry, in the form of a sum of elementary concentrations associated with concentration profiles defined on manifolds induced by a limiting procedure at infinity, and thus different from $M$. The profiles satisfy an inequality of Plancherel type: the sum of the quadratic forms of Laplace-Beltrami operators for the profiles on their respective manifolds is bounded by the quadratic form of the Laplace-Beltrami operator of the sequence. A similar relation, related to the Brezis-Lieb Lemma, holds for the $L^p$-norms of profiles on the respective manifolds. 
\end{abstract}


\maketitle

\section{Introduction}

\emph{Defect of compactness}, relative to an embedding of two Banach
spaces $E\hookrightarrow F$, is a difference $u_{k}-u$ between a
weakly convergent sequence $u_{k}\rightharpoonup u$ in $E$ and its
weak limit, taken up to a suitable remainder that vanishes in the norm of $F$.
In particular, if the embedding is compact and $E$ is reflexive,
the defect of compactness is null. For many embeddings there exist
well-structured representations of the defect of compactness, known
as \emph{profile decompositions}. Best studied are profile decompositions
relative to Sobolev embeddings, which are sums of terms with asymptotically
disjoint supports, called \emph{elementary concentrations}
or \emph{bubbles}. Profile decompositions were originally motivated
by studies of concentration phenomena in PDE in the early 1980's by Uhlenbeck,
Brezis, Coron, Nirenberg, Aubin and Lions, and they play significant
role in verification of convergence of functional sequences in applied
analysis, particularly when the information available via the classical
concentration compactness method is not enough detailed. 

Profile decompositions are known to exist when the embedding $E\hookrightarrow F$ is \emph{cocompact} relative to some group $\mathcal{G}$ of bijective
isometries on $E$. An embedding $E\hookrightarrow F$ is called $\mathcal{G}$-cocompact if any sequence $(u_{k})$ in $E$ satisfying $g_{k}u_{k}\rightharpoonup0$
for any sequence of operators $(g_{k})$ in $\mathcal{G}$ vanishes
in the norm of $F$. (It is easy to verify, for example, that $\ell^{\infty}(\Z)$ is cocompactly embedded into itself relative to the group of shifts $\mathcal{G}=\lbrace(a_{n})\mapsto(a_{n+m})\rbrace_{m\in\Z}$.) The
earliest cocompactness result for functional spaces known to the authors
is the proof of cocompactness of embedding of the inhomogeneous Sobolev
space $H^{1,p}(\R^{N})$, $N>p$, into $L^{q}$, $q\in(p,p^{*})$,
where $p^{*}=\frac{pN}{N-p}$, relative to the group of shifts $u\mapsto u(\cdot-y),\; y\in\R^{N}$,
by E. Lieb \cite{LIeb} (the term \emph{cocompactness} itself appeared
in literature only the last decade). A profile decomposition relative to a group $\mathcal{G}$ of bijective isometries represents defect of compactness as a sum of \emph{elementary
concentrations}, or \emph{bubbles}, $\sum_{n\in\N}g_{k}^{(n)}w^{(n)}$
 with some $g_{k}^{(n)}\in\mathcal{G}$ and $w^{(n)}\in E$, $k\in\N$,
$n\in\N$. The elements $w^{(n)}$, called \emph{concentration profiles},
are then obtained as weak limits of $(g_{k}^{(n)})^{-1}u_{k}$ as
$k\to\infty$. Typical examples of groups $\mathcal{G}$, involved
in profile decompositions, are the above mentioned group of shifts
and the rescaling group, which is a product group of shifts and dilations
$u\mapsto t^{r}u(t\cdot)$, $t>0$, where $r=\frac{N-p}{p}$ for $\dot{H}^{1,p}(\R^{N})$,
$N>p$. 

Existence of profile decompositions for general bounded sequences
in $\dot{H}^{1,p}(\R^{N})$ equipped with the rescaling group was
proved by Solimini \cite{Solimini}, and later, independently, but
with a weaker form of asymptotics, in \cite{Gerard} and \cite{Jaffard}
(\cite{Jaffard} also extended the result to fractional Sobolev spaces).
It was first observed in \cite{SchinTin} that profile decomposition
(and thus concentration phenomena in general) can be understood in
functional-analytic terms, rather than in specific function spaces.
The result of \cite{SchinTin} was extended in \cite{SoliTi} to uniformly
convex Banach spaces with the Opial condition (without the Opial condition
profile decomposition still exists but in terms of the less-known
Delta convergence instead of weak convergence). However, despite the
general character of the statement in \cite{SoliTi}, it does not apply to several known profile decompositions, in particular, when the space $E$
is not reflexive (e.g. \cite{AT_BV}), when one has only a semigroup
of isometries (e.g. \cite{AdiTi-Pisa}), or when the profile decomposition
can be expressed without a group (e.g. Struwe \cite{Struwe}). 

The present paper follows the direction started by the work of Struwe,
to study profile decompositions in the Sobolev space of a non-compact Riemannian manifold that possibly \emph{lacks} a nontrivial isometry group. When the isometry group $\mathrm{Iso}(M)$ of manifold $M$ is sufficiently rich, namely, if 
\begin{equation}
M=\bigcup_{\eta\in\mathrm{Iso}(M)}\eta K\mbox{ for some compact set }K\subset M,\label{eq:coco}
\end{equation}
it is shown in \cite{FiesTin} that Sobolev embedding $H^{1,2}(M)\hookrightarrow L^{p}(M)$,
$2<p<\frac{2N}{N-2}$, $N\ge2$, becomes cocompact relative to the
action of $\mathrm{Iso}(M)$. In this case a profile decomposition is immediate from the functional-analytic statement of \cite{SchinTin}.

In what follows we use the standard invariant norm of $H^{1,2}(M)$, $\|u\|_{1,2}=\left(\int_{M}(|du|^{2}+u^{2})\dv\right)^{1/2}$, where $dv_g$  is the Riemannian measure on $M$,
and we always assume that $N\ge 2$. We quote the result of \cite{FiesTin},
with the property of unconditional convergence added from the general
profile decomposition in $ $\cite{SoliTi}.
\begin{thm}
\label{thm:FT} Let $M$ be a complete Riemannian manifold with a
countable group $G$ of isometries satisfying (\ref{eq:coco}), and
let $(u_{k})$ be a bounded sequence in $H^{1,2}(M)$. Then there
exists $w^{(n)}\in H$, $g_{k}^{(n)}\in G$, $k,n\in\N$, such that
for a renumbered subsequence 
\begin{gather}
g_{k}^{(1)}=id,\;({g_{k}^{(n)}}^{-1}g_{k}^{(m)})_{k}\mbox{ is discrete for }n\neq m,\label{separates}\\
w^{(n)}=\wlim u_{k}\circ g_{k}^{(n)}\label{eq:profile}\\
\sum_{n\in{\N}}\|w^{(n)}\|_{1,2}^{2}\le\limsup\|u_{k}\|_{1,2}^{2}\label{norms}\\
u_{k}-\sum_{n\in{\N}}w^{(n)}\circ{g_{k}^{(n)}}^{-1}\to0\mbox{ in }L^{p}(M),\;2<p<2^{*},\label{BBasymptotics}
\end{gather}
and the series $\sum_{n\in{\N}}w\circ g_{k}^{(n)}$ converges unconditionally
and uniformly with respect to $k$.
\end{thm}
In particular, (\ref{eq:coco}) holds, implying the assertion of the
theorem, when $\mathrm{Iso}(M)$ is transitive, i.e. $M$ is homogeneous
space, e.g. if $M$ is $\R^{N}$ or the hyperbolic space
$\mathbb{H}^{N}$. When a non-compact manifold $M$ has no nontrivial
isometries, it does not of course  mean that the Sobolev
embedding $H^{1,2}(M)\hookrightarrow L^{p}(M)$, $2<p<2^{*}$ is compact,
as we demonstrate in the Example \ref{exa:noncompact} below. Thus
the question remains if one can express the corresponding defect of
compactness in a form similar to profile decomposition of similar to \eqref{BBasymptotics}. In this
paper we answer this question positively for manifolds of bounded geometry,
as defined below. Absence of a group of isometries comes, however at some cost, which is transparent already from Struwe's profile decomposition in \cite{Struwe}, where profiles are functions on the tangent space of $M$ at the points of concentration: in general, absence of a non-compact group $\mathcal G$ of isometries that may produce blowup sequences of the form $g_kw\rightharpoonup 0$, $g_k\in\mathcal G$ corresponds to emergence of concentration profiles $w^{(n)}$ supported on metric structures  different from $M$.  This is indeed the case in the present paper that deals with profile decomposition relative to the embedding $H^{1,2}(M)\hookrightarrow L^{p}(M)$ when $M$ is a Riemannian manifold of bounded geometry. 

The subject of the paper was proposed to one of the authors a number of years ago by Richard Schoen \cite{Schoen}.

The paper is organized as follows. In Section 2 we give an analog
of the cocompactness property expressed without invoking the isometry group, in terms of the ``spotlight vanishing'' Lemma \ref{lem:spotlight},
which naturally requires the manifold to have bounded geometry. This
lemma motivates our construction of profile decomposition in the main
result of the paper, Theorem \ref{thm:main}, based on patching of
local profiles moving along the manifold. In Section 3 we define the
manifolds at infinity needed to formulate  Theorem
\ref{thm:main}. Manifolds at infinity play the same role in description
of elementary concentrations based on quasi-translations as the tangent
space plays in the descriptions of elementary concentrations based
on dilations in \cite{Struwe}. In
Section 4 we state the main result, as well as provide construction of global profiles as functions on the manifolds at infinity, rather than on the manifold $M$ itself.
Section 5 contains technical statements concerning reconstruction
of the original sequence from its local profiles. Proof of Theorem
\ref{thm:main} is given in the Section 6. In Section 7 we show that
if $M$ satisfies (\ref{eq:coco}), then Theorem
\ref{thm:FT} is a particular case of Theorem \ref{thm:main}.
Appendix contains some elementary properties of manifolds of bounded
geometry, existence of a suitable uniform covering, and a gluing theorem
used in the construction of manifolds at infinity.

\section{A ``spotlight'' lemma and preliminary discussion}

Let $M$  be a smooth, complete $N$-dimensional Riemannian manifold with metric $g$ and a positive injectivity radius $r(M)$. 
In what follows $B(x,r)$ will denote a geodesic ball in $M$ and
$\Omega_{r}$ will denote the ball in $\R^{N}$ of radius $r$ centered at the origin. Let $r\in(0,r(M))$ be fixed. Then the Riemannian exponential map $\mathrm{exp}_{x}$ is a diffeomorphism of $\{v\in T_xM:\, g_x(v,v)< r\}$ onto $B(x,r)$.  For each $x\in M$ we choose an orthonormal basis for $T_xM$ which yields an identification $i_x:\R^N \rightarrow T_xM$.  Then  $e_{x}:\Omega_{r}\to B(x,r)$ will denote  geodesic normal coordinates at $x$ given by $e_{x}=\mathrm{exp}_{x}\circ i_x$. We do not require smoothness of the map $i_x$ with respect to $x$, since in the arguments $x$ will be taken from a discrete subset of $M$.

From now on we assume that $M$ is
a connected non-compact manifold of bounded geometry.  The latter is defined as follows, e.g. cf. \cite{Shubin}.  

\begin{defn}
\label{def:bg} 
 A smooth Riemannian
manifold $M$ is of bounded geometry if the following two conditions
are satisfied: 

(i) The injectivity radius $r(M)$ of $M$ is positive. 

(ii) Every covariant derivative of the Riemann curvature tensor $R^{M}$of
M is bounded, i.e., $\nabla^{k}R^{M}\in L^{\infty}(M)$ for every
$k=0,1,\dots$ 
\end{defn}
Please note that a Riemannian  manifold of bounded geometry is always  complete.   On every paracompact manifold $M$ one can define  a Riemannian metric tensor $g$ such that $(M,g)$ is a manifold of bounded geometry, cf. \cite{Green}. We refer the reader to the appendix for elementary properties of manifolds of bounded geometry used in this paper. Here we recall only the notion of the discretization of the manifold that is crucial for our constructions.  

\begin{defn}
\label{def:discr}  A subset $Y$ of  Riemannian manifold $M$ is called  $\varepsilon$-discretization of $M$, $\varepsilon>0$,  if the distance between any two distinct points of $Y$ is greater than or equal to $\varepsilon$ and 
\[ M = \bigcup_{y\in Y} B(y,\varepsilon).\]
\end{defn}
Any connected Riemannian manifold $M$ has a $\varepsilon$-discretizations for any $\varepsilon>0$, and if $M$ is of bounded geometry then for any $t\ge  1$ the covering  $\{ B(y,t\varepsilon)|\}_{y\in Y}$ is uniformly locally finite, cf. Lemma \ref{lem:covering}.   

\begin{example}
\label{exa:noncompact}Let $M$ be a non-compact manifold of bounded geometry,
let $w\in C_{0}^{1}(\Omega_{r})\setminus\{0\}$, let $(x_{k})$ be a discrete sequence
on $M$, and let $u_{k}=w\circ e_{x_{k}}^{-1}$. Then it is easy to
see that $u_{k}\rightharpoonup0$ while $\|u_{k}\|_{p}$ is bounded
away from zero by (\ref{eq:peq}). In other words, for non-compact manifolds of
bounded geometry presence of a \emph{local} concentration profile
$w$ results in a nontrivial defect of compactness.
\end{example}
The main result of the paper, Theorem \ref{thm:main}, is an analog
of Theorem \ref{thm:FT} based on local concentration profiles in
the spirit of Example \ref{exa:noncompact}. Once we subtract from
the sequence all suitably patched local ``runaway bumps'' of the
form $w\circ e_{y_{k}}^{-1}$, the remainder sequence $(v_{k})$ is
expected to have no nonzero local profiles left, in other word, to
satisfy $v_{k}\circ e_{y_{k}}\rightharpoonup0$ in $H^{1,2}(\Omega_{\rho})$
with some $\rho>0$. This is a condition related to the one in the
cocompactness Lemma~2.6 of \cite{FiesTin}, and it implies that
$(v_{k})$ vanishes in $L^{p}(M)$. In strict terms we have the following
``spotlight vanishing'' lemma. 
In what follows $2^*$ denotes the  Sobolev conjugate of $2$ i.e. $\frac{1}{2^*}=\frac{1}{2}- \frac{1}{N}$.    

\begin{lem}["Spotlight lemma"]
\label{lem:spotlight}Let $M$ be an $N$-dimensional Riemannian manifold of bounded geometry and let $Y\subset M$ be a $r$-discretization of $M$,  $r<r(M)$. 
 Let  $(u_{k})$ be a bounded sequence in $H^{1,2}(M)$.  Then, $u_{k} \to 0$ in $L^{p}(M)$ for any $p\in(2,2^{*})$ if and only if  
 $u_{k}\circ e_{y_{k}}\rightharpoonup 0$ in
$H^{1,2}(\Omega_{r})$  for any  sequence $(y_k)$, $y_{k}\in Y$.  
\end{lem}

\begin{proof}
Let us fix $p\in(2,2^{*})$ and assume that $u_{k}\circ e_{y_{k}}\rightharpoonup 0$ in
$H^{1,2}(\Omega_{r})$  for any  sequence $(y_k)$, $y_{k}\in Y$.   The local Sobolev embedding theorem and the boundedness of the geometry of $M$ implies that there  exists
$C>0$ independent of $y\in M$ such that 
\[
\int_{B(y,r)}|u_{n}|^{p}\dv\le C\int_{B(y,r)}(|\nabla u_{n}|^{2}+|u_{n}|^{2})\dv\left(\int_{B(y,r)}|u_{n}|^{p}\dv\right)^{1-2/p}.
\]
Adding the terms in  left and  right hand side over $y\in Y$ 
we have 
\begin{equation}
\int_{M}|u_{n}|^{p}\dv\le C\int_{M}(|\nabla u_{n}|^{2}+|u_{n}|^{2})\dv\; \sup_{y\in Y}\left(\int_{B(y,r)}|u_{n}|^{p}\dv\right)^{1-2/p}.\label{eq:intro2}
\end{equation}
Boundedness of the sequence  $(u_{n})$  in $H^{1,2}(M)$ implies that the supremum of the right hand side is finite. So for any $u_{k}$ we can find a sequence  
$y_{k}\in Y$, $k\in\N$, 
such that 
\begin{equation}
\sup_{y\in Y}\int_{B(y,r)}|u_{k}|^{p}\dv \le 2^{\frac{p}{p-2}}\int_{B(y_{k},r)}|u_{k}|^{p}\dv.\label{eq:intro1}
\end{equation}
By compactness of the Sobolev embedding $H^{1,2}(\Omega_{r})\hookrightarrow L^{p}(\Omega_{r})$ and weak convergence of the sequence in $H^{1,2}(\Omega_{r})$ we have $u_{k}\circ e_{y_{k}}\to 0$ in  $L^{p}(\Omega_{r})$, and thus,  $\int_{B(y_{k},r)}|u_{k}|^{p}\dv \to 0$. Combining this with (\ref{eq:intro2}) and (\ref{eq:intro1}) we have $u_{k} \to 0$ in $L^{p}(M)$. 

Assume now that $u_{k} \to 0$ in $L^{p}(M)$. Boundedness of the geometry of $M$ implies  for any sequence $(y_k)$ that  $u_{k}\circ e_{y_k} \to 0$ in $L^{p}(\Omega_r)$.  On the other hand  boundedness of the sequence $u_k$ in $H^{1,2}(M)$ and boundedness of geometry give us boundedness of any sequence   $(u_{k}\circ e_{y_k})$  in $H^{1,2}(\Omega_r)$. By continuity of the embedding 
$H^{1,2}(\Omega_{r})\hookrightarrow L^p(\Omega_r)$  we get  $u_{k}\circ e_{y_{k}}\rightharpoonup 0$ in $H^{1,2}(\Omega_{r})$.
\end{proof}

The main result of this paper, Theorem \ref{thm:main}, requires a
definition of a manifold at infinity of $M$ 
associated with a given discrete sequence $(y_{k})$ in $M$, as well as  a proof that such manifold exists. These are given in Section 3. Thus we dedicate
the rest of introduction to discussing the place of our settings (subcritical
Sobolev embedding, manifold of bounded geometry) in the context of
existing or possible results concerning profile decompositions in
Sobolev spaces of Riemannian manifolds. 

Struwe \cite{Struwe} (see also the exposition in the book \cite{DHR})
gives a profile decomposition for the limiting case $p=2^{*}$ of
the Sobolev embedding on a compact manifold. By means of a finite
partition of unity and the exponential map this profile decomposition
follows from the profile decomposition for the limiting Sobolev embedding
for the case of a bounded domain in $\R^{N}$. This, in turn, is a
consequence of the profile decomposition for the embedding $\dot{H}^{1,2}(\R^{N})\hookrightarrow L^{2^{*}}(\R^{N})$
based on the rescaling group which is a product group of shifts $u\mapsto u(\cdot-y)$,
$y\in\R^{N}$, and dilations $u\mapsto t^{\frac{N-2}{2}}u(t\cdot)$,
$t>0$. However, for sequences supported in a bounded domain of $\R^{N}$
profile decomposition cannot contain shifts to infinity or deflations
$u\mapsto t_{k}^{\frac{N-2}{2}}u(t_{k}\cdot)$, $t_{k}\to0$, or superpositions
thereof, so it consists only of blowup terms $u\mapsto t_{k}^{\frac{N-2}{2}}u(t_{k}\cdot)$,
$t_{k}\to\infty$, with bounded (or, equivalently, modulo vanishing remainder,
constant) shifts. 

By analogy with the case $M=\R^{N}$, one could expect that generalizing
Struwe's profile decomposition to a non-compact manifold would mean
finding a way to express loss of compactness with respect to shifts
along the manifold in combination with changes of scale responsible
for loss of compactness in the limiting case $p=2^{*}$. While one
can easily define a blowup of a local profile traveling along points
$y_{k}\in M$ as $x\mapsto t_{k}^{\frac{N-2}{2}}w(t_{k}e_{y_{k}}^{-1}(x))$
by $t_{k}^{-\frac{N-2}{2}}u_{k}(e_{y_{k}}(t_{k}^{-1}\cdot))\rightharpoonup w$
in $H^{1,2}(\Omega_{r})$ with $t_{k}\to\infty$, this construction
does not extend to the opposite end of scale, i.e. $t_{k}\to0$ and
has no simple counterpart in the non-Euclidean case:
a putative deflating transformation must be substantially dependent on the geometry of the manifold at every point.

In this paper we provide a profile decomposition only for subcritical
Sobolev embeddings, which in the Euclidean case involve only the group
of shifts. We use the exponential map to define a \emph{local} counterpart
of translations ``along'' a sequence of points $y_{k}\in M$, namely,
a ``spotlight'' sequence $u_{k}\circ e_{y_{k}}:\Omega_{r}\to B(y_{k},r)$.
Like in \cite{Struwe}, reconstruction of the original sequence from its concentration profiles
involves patching the (local) profiles, composed with the inversed
exponential map, by a partition of unity on $M$. 

Without the assumption of bounded geometry, bounded sequences in
$H^{1,2}(M)$ do not admit, in general, a profile decomposition for
the mere reason that there might be no embedding $H^{1,2}(M)\hookrightarrow L^{p}(M)$ except for the trivial case $p=2$. Even if the embedding exists, but the geometry is not bounded, local translations along the manifold may
induce complicated - nonlinear and anisotropic - changes of scale,
which are likely to affect the expression for the defect of compactness.
The critical case $p=2^{*}$ of the problem has to cope not only with
this difficulty, as well as with the already mentioned
issue of additional loss of compactness due a putative non-Euclidean analog of deflations (the opposite end of scale to blowups) in the Euclidean space.

\section{Manifolds at infinity}

In what follows we consider the radius $\rho<\frac{r(M)}{8}$ and $\hat{\rho}$-discretization  $Y$ of $M$, $\frac{\rho}{2}<\hat{\rho}<\rho$. In what follows we will use
the notation $\N_{0}\eqdef\N\cup\lbrace0\rbrace$.
\begin{defn}
\label{def:trailing} Let $(y_{k})_{k\in\N}$ be a sequence in $Y$ that is an enumeration of the infinite subset of $Y$.
A countable family $\lbrace(y_{k;i})_{k\in\N}\rbrace_{i\in\N_{0}}$
of sequences on $Y$ is called  \emph{a  trailing system} for $(y_{k})_{k\in\N}$ if for every $k\in\N$ $(y_{k;i})_{i\in\N_{0}}$
is an ordering of $Y$ by the distance from $y_{k}$, that is, an
enumeration of $Y$ such that $d(y_{k;i},y_k)\le d(y_{k;i+1},y_k)$ for
all $i\in\N_{0}$. In particular, $y_{k;0}=y_k$. 
\end{defn}

Note that any enumeration of the infinite subset of  $Y$ admits a trailing system: it can be constructed inductively, by starting with $y_{k;0}=y_{k}$ and, given $i\in\N_{0}$, choosing $y_{k;i+1}$ as any point $y\in Y\setminus\lbrace y_{k;0},\dots,y_{k;i}\rbrace$ with the least value of $d(y,y_{k})$, $i\in\N_{0}$. The trailing system is generally not uniquely defined when for some $k\in\N$ there are several points of $Y$ with the same distance from $y_{k}$.

\begin{lem}\label{Ji}
Let $(y_{k})_{k\in\N}$ be a sequence in a discretization $Y$ that is an enumeration of the infinite subset of $Y$.
There exists a renamed subsequence of $(y_k)_{k\in\N}$ with the following property: for any $i\in \N_0$ there exist a finite subset $J_i$ of $\N_0$ such that  
\begin{equation}\label{Ji0}
B(y_{k;i},\rho)\cap B(y_{k;j},\rho)\not=\emptyset \; \Longleftrightarrow\;  j\in J_i \, . 
\end{equation}
\end{lem}
\begin{proof}
Let us fix $i$. If the ball $B(y_{k;j},\rho)$ intersects $B(y_{k;i},\rho)$ then $B(y_{k;\ell },\rho/2)\subset B(y_{k}, d(y_k,y_{k;i})+3\rho)$ for any $\ell\in \{0,1,\ldots , j\}$. The geometry of $M$ is bounded so the respective volumes of the balls $(B(y_{k;\ell },\rho/4)$ are bounded from bellow by a constant depending on $\rho$ but independent of the  balls. Note that these balls are pairwise disjoint. Moreover the Ricci curvature of $M$ is bounded from below, so  by the Bishop-Gromov volume comparison  theorem  the volume of any ball $B(y_{k;\ell },r)$ can be estimated from above by the constant depending only on the radius. In consequence  
\begin{equation}\label{Ji1} 
 C\, j\; \le \sum_{\ell=0}^j  \vol(B(y_{k;\ell },\rho/4) \le  \vol \Big(B\big(y_{k}, d(y_k,y_{k;i})+3\rho\big)\Big) \le C_i,  
 \end{equation}
and the constant $C_i$ is  independent of $k$.  
Let $J_{k;i}= \{j:  B(y_{k;i},\rho)\cap B(y_{k;j},\rho)\not=\emptyset  \}$. Then for any $k$ we have $J_{k;i}\subset [0, C_i/C]$. Therefore there exists a subsequence  $k_1,k_2,\cdots$ such that $J_{k_\ell,i}=J_{k_\nu,i}$ for any $\ell$ and $\nu$. We put $J_{i}=J_{k_1,i}$.   
 
The assertion of the lemma follows now from the standard diagonalization argument.  
\end{proof}

We will always assume throughout the paper that the sequence we work with satisfies the above property. This can be done since passing to subsequence never spoils our construction.

With a given trailing system $\lbrace(y_{k;i})_{k\in\N}\rbrace_{i\in\N_{0}}$ we associate a manifold $M_\infty^{(y_{k;i})}$ defined by gluing data that will be constructed below. In the construction we will use definitions from the second part of  Appendix.

When we define the manifold $M_\infty^{(y_{k;i})}$ we assume that we work with a sequence satisfying \eqref{Ji0}.   The following subset of $\N^2_0$ is essential for the construction:
\[ \mathcal{K} = \bigcup_{i=0}^\infty \{(i,j):\; j\in J_i \} . \]
If $(i,j)\in \mathcal{K}$, then passing to a subsequence  for any  $\xi,\eta\in\Omega_{2\rho}$ we have
\[
d(e_{y_{k;j}}\xi,e_{y_{k;i}}\eta)\le d(e_{y_{k;j}}\xi,y_{k;j})+d(y_{k;j},y_{k;i})+d(y_{k;i},e_{y_{k;i}}\eta)< 6\rho<\frac{3r(M)}{4}\, .
\]
Therefore, on a subsequence, we may consider a diffeomorphism 
\[
\psi_{ij,k}\eqdef e_{y_{k;i}}^{-1}\circ e_{y_{k;j}}:\;\bar{\Omega}_{2\rho}\to\Omega_{a} ,\qquad a=\frac 3 4 r(M) . 
\]

To each pair $(i,j)\in\mathcal{K}$ we associate a subset $\Omega_{ji}$ of $\Omega_{2\rho}$ and  a diffeomorphism  $\psi_{ij}$ defined on $\Omega_{ji}$ whenever the latter is nonempty. 
By boundedness of the geometry, cf. Lemma \ref{lem:bdd-derivatives}, and the Ascoli-Arzela theorem,
there is a renamed subsequence of $(\psi_{ij,k})_{k\in\N}$ that converges
in $C^{\infty}(\bar{\Omega}_{2\rho})$ to some smooth function $\psi_{ij}:\bar{\Omega}_{2\rho}\to\Omega_{a}$,
and, moreover, we may assume that the same extraction of $(\psi_{ji,k})_{k\in\N}$
converges in $C^{\infty}(\bar{\Omega}_{2\rho})$ as well. Note that
Lemma \ref{lem:bdd-derivatives} gives that for any $\alpha\in\N_{0}^{N}$
there exists a constant $C_{\alpha}>0$, such that 
\[
| D^{\alpha}\psi_{ij}(\xi)|\le C_{\alpha}\mbox{ whenever }i,j\in\N_{0},\;\xi\in\Omega_{\rho}.
\]
We define $\Omega_{ij}\eqdef\psi_{ij}(\Omega_{\rho})\cap\Omega_{\rho}$.
This set may generally be empty. Let us define a set that we will invoke in our application of Corollary \ref{cor:gluing} that will follow: 
\begin{equation}
\mathbb{K}\eqdef \{(i,j)\in \mathcal K:\;\Omega_{ij}\neq\emptyset\}.
\end{equation}
To prove the cocycle condition for the gluing data we should extract subsequences in  a more restrictive way. First we consider a
subsequence $\psi_{01,k}^{1}$ of $\psi_{01,k}$ that converges to
$\psi_{01}$ and note that on the same subsequence we have convergence
of $\psi_{10,k}^{1}$ to $\psi_{10}$. Fix an enumeration $n\mapsto(i_{n},j_{n})$
of the set of all indices $(i,j)\in\mathbb{K}$, 
$i<j$, and extract the
convergent subsequence $\psi_{i_{\ell}j_{\ell},k}^{n+1}$ of the subsequence
$\psi_{i_{\ell}j_{\ell},k}^{n}$ from the previous extraction step,
for $\ell=0,\dots,n+1$. Then the diagonal sequence $\psi_{i_{\ell}j_{\ell},k}^{k}$
will converge to $\psi_{i_{\ell}j_{\ell}}$ for any $\ell\in\N$.

By the definition of $\Omega_{ij}$ and $\psi_{ij}$ we have  $\psi_{ij}\circ\psi_{ji}=\mathrm{id}$ on $\Omega_{ij}$ and $\psi_{ji}\circ\psi_{ij}=\mathrm{id}$ on $\Omega_{ji}$. Therefore $\psi_{ji}=\psi_{ij}^{-1}$  in restriction to $\Omega_{ij}$, and $\psi_{ji}$ is a diffeomorphism between $\Omega_{ij}$ and $\Omega_{ji}$. Note that this construction gives that $\psi_{ii}=\mathrm{id}$ , $\Omega_{ii}=\Omega_{\rho}$ for all $i\in\N_{0}$.  Thus conditions \emph{(i-iii)} of Corollary \ref{cor:gluing} are satisfied. 

Note also  that the second step of the constructions implies  
\begin{align*}
\psi_{\ell i} = &\lim_{k\to\infty}e_{y_{k;\ell}}^{-1}\circ e_{y_{k;i}} = 
\lim_{k\to\infty}e_{y_{k;\ell}}^{-1}\circ e_{y_{k;j}}\circ e_{y_{k;j}}^{-1}\circ e_{y_{k;i}}= \\
& \lim_{k\to\infty}e_{y_{k;\ell}}^{-1}\circ e_{y_{k;j}}\circ\lim_{k\to\infty}e_{y_{k;j}}^{-1}\circ e_{y_{k;i}}=\psi_{\ell j}\circ\psi_{ji},
\end{align*}
and
\[
\psi_{ij}(\Omega_{ji}\cap\Omega_{jk})=\psi_{ij}(\psi_{ji}(\Omega_{\rho})\cap\Omega_{\text{\ensuremath{\rho}}}\cap\psi_{jk}(\Omega_{\rho})\cap\Omega_{\text{\ensuremath{\rho}}})=\Omega_{ij}\cap\Omega_{ik},
\]
which proves condition \emph{(iv)} of Corollary \ref{cor:gluing}. 

Let $x\in \partial \Omega_{ij}\cap \Omega_{\rho}$. Since $\partial \Omega_{ij}\subset \partial \psi_{ij}(\Omega_{\rho})\cup \partial\Omega_{\rho}$ and $\Omega_{\rho}$ is open we conclude that $x\in \partial \psi_{ij}(\Omega_{\rho})=  \psi_{ij}( \partial \Omega_{\rho})$. Thus $\psi_{ji}(x)\in \partial \Omega_{\rho}$. This proves the condition \emph{(v)} of Corollary \ref{cor:gluing}.  


We have thus proved the following proposition, cf.  Corollary \ref{cor:gluing}.
\begin{prop}
Let $M$ be a Riemannian manifold with bounded geometry and let $Y$ be its discretization. 

For any trailing system $\lbrace(y_{k;i})_{k\in\N}\rbrace_{i\in\N_{0}}$ related to the sequence $(y_k)$  in $Y$ there exists a smooth manifold $M_{\infty}^{(y_{k;i})}$  with an atlas $\{(U_i,\tau_i)\}_{i\in\N_{0}}$ such that:\\
1)   $\tau_i(U_i)=\Omega_\rho$, \\
and\\
2)  there exists a renamed subsequence of $k$ such that for any two charts  $(U_i,\tau_i)$ and $(U_i,\tau_i)$ with $U_i\cap U_j\not=\emptyset$  the corresponding  transition map $\psi_{ij}: \tau_j(U_j\cap U_i)\rightarrow \tau_i(U_j\cap U_i)$ is given by the $C^\infty$-limit
\[
\psi_{ij}=\lim_{k\rightarrow \infty} e_{y_{k;i}}^{-1}\circ e_{y_{k;j}}. 
\]
\end{prop}

For convenience we will also widely use the "inverse" charts $\varphi_i= \tau_i^{-1}$ so that $\varphi_{j}^{-1}\circ\varphi_{i}=\psi_{ji}:\Omega_{ij}\to\Omega_{ji}$.

Now define the Riemannian metric on $M_{\infty}^{(y_{k;i})}$ in two steps as follows. First for any $i\in \N_0$ we define a metric tensor $g^{(i)}$ on $\Omega_\rho$ and afterwards we pull it back onto $U_i=\varphi_i(\Omega_\rho)\subset M_{\infty}^{(y_{k;i})}$ via $\varphi_i^{-1}$ and prove the compatibility conditions.   
 
Tensor $g^{(i)}$ is defined as a $C^\infty$-limit on a suitable renamed subsequence:
\begin{equation}
\widetilde{g}_{\xi}^{(i)}(v,w)\eqdef\lim_{k\to\infty} g_{e_{y_{k;i}(\xi)}}   \big(d e_{y_{k;i}}(v), d e_{y_{k;i}}(w)\big) , \; \xi\in \Omega_\rho \; \text{and}\; v,w\in \R^N,\label{eq:pre-metric}
\end{equation}
Existence of the limit follows from the boundedness of the geometry of the manifold $M$ since the coefficients of the tensors $g_{e_{y_{k;i}}}$ form a bounded family of functions in the spaces $C^\infty(\Omega_\rho)$.
Using the standard diagonalization procedure  we can choose the same  subsequence for any $i$.  Furthermore, $\widetilde{g}^{(i)}$  is a bilinear symmetric positive-definite form. Since we used in the definition \eqref{eq:pre-metric}  normal coordinates, we have $\widetilde{g}_{0}^{(i)}(v,v)=|v|^2$. In consequence, 
by the boundedness  of geometry, $\widetilde{g}^{(i)}_{\xi}[v,v]\ge\frac{1}{2}|v|^2$ in $\Omega_{\rho}$ for all $i\in\N_0$, provided that $\rho$ is fixed sufficiently small.

Now we can define   a metric $\widetilde{g}$ on $M_\infty^{(y_{k;i})}$ by the following relation 
\begin{align} \label{eq:metric}
\widetilde{g}_{x}(v,w)  &\eqdef \widetilde{g}^{(i)}_{\varphi_i^{-1}(x)}\big(d \varphi_i^{-1}(v), d \varphi_i^{-1}(w)\big) , \\ 
\qquad &  x\in \varphi_{i}(\Omega_\rho) \subset M_{\infty}^{(y_{k;i})}  \; \text{and}\; v,w\in T_x M_{\infty}^{(y_{k;i})} . \nonumber 
\end{align}
To prove that the Riemannian metric is well defined we should  verify the compatibility relation on overlapping charts, i.e. that 
\begin{align} \label{eq:metric_1.2}
\widetilde{g}^{(i)}_{\varphi_i^{-1}(x)}(d \varphi_i^{-1} v,  & d \varphi_i^{-1} w)  =  \widetilde{g}^{(j)}_{\varphi_j^{-1}(x)}(d \varphi_j^{-1} v, d \varphi_i^{-1} w) , \\ 
\qquad &\text{if}\qquad   x\in \varphi_{i}(\Omega_\rho)\cap \varphi_{j}(\Omega_\rho)  \; \text{and}\; v,w\in T_x M_{\infty}^{(y_{k;i})} . \nonumber 
\end{align}
But $\varphi^{-1}_j\circ \varphi_i =\psi_{ji}$, so it suffices to prove that
\begin{align} \label{eq:metric_1.2a}
\widetilde{g}^{(i)}_{\xi}( v,  w)  =  \widetilde{g}^{(j)}_{\psi_{ji}(\xi)}(d \psi_{ji} v, d \psi_{ji} w) , \qquad  
\text{with}\qquad   v,w\in T_\xi \Omega_\rho.  
\end{align}
Let $e_{y_{k;j}}^{-1}\circ e_{y_{k;i}} (\xi) = \eta_k  $ then $\psi_{j,i}(\xi)= \lim_{k\rightarrow \infty} \eta_k$  and $e_{y_{k;i}}(\xi) = e_{y_{k;j}}(\eta_k)$. In consequence
\begin{align} \label{eq:metric_1.2b}
\widetilde{g}^{(i)}_{\xi}( v,  w) & =  \lim_{k\rightarrow \infty} g_{e_{y_{k;i}}(\xi)}(d e_{y_{k;i}}  v, d e_{y_{k;i}} w)  = \\ 
\qquad & = \lim_{k\rightarrow \infty} g_{e_{y_{k;j}}(\eta_k)}(d e_{y_{k;j}^{-1}}e_{y_{k;i}}  v, d e_{y_{k;j}^{-1}}\circ e_{y_{k;i}} w)  = \nonumber \\ 
\qquad & =  g_{\psi_{j,i}(\xi)}(d \psi_{ji} v, d \psi_{ji}  w)   \nonumber
\end{align}

\begin{defn}
A manifold at infinity $M_{\infty}^{(y_{k;i})}$ of a manifold $M$
with bounded geometry, generated by a trailing system $\lbrace(y_{k;i})_{k\in\N}\rbrace_{i\in\N_{0}}$ of a sequence $(y_k)$ in $Y$,
is the differentiable manifold given by Theorem \ref{cor:gluing}, supplied
with a Riemannian metric tensor $\widetilde{g}$ defined  by (\ref{eq:metric}).   
\end{defn}

For the given  chart $(\Omega_\rho, \tau_i)$ components of the metric tensor $\widetilde{g}$ are defined  by  formula \eqref{eq:pre-metric}, cf. \eqref{eq:metric}. Let $\xi=0$. The maps $e_{y_{k;i}}$ are 
 normal coordinates systems, so for any $k$  components $g_{\ell,m}$ of the metric tensor $g$ satisfy  $g_{\ell,m}(0)=\delta_{\ell, m}$ and $\partial_ng_{\ell,m}(0)=0$. So by identity \eqref{eq:pre-metric} we get 
\[
\widetilde{g}_{\ell,m}(0)=\delta_{\ell, m} \qquad \text{and}\qquad \partial_n \widetilde{g}_{\ell,m}(0)= 0 \, .
\]
Moreover the components $g_{\ell,m}$ are a bounded set  in $C^{\infty}(\Omega_\rho)$ so all the set  of $\widetilde{g}_{\ell,m}$ is also  bounded in $C^{\infty}(\Omega_\rho)$.

For any $k$ and  $i$,  $(\Omega_\rho, e_{y_{k;i}})$ is a normal coordinate system, so for any unit vector v we have on that ball $\Gamma_{m,\ell}^n(tv)v_\ell v_m=0$, $0\le t\le\rho$, where $\Gamma_{m,\ell}^n$ denotes  Christoffel symbols of a given Riemannian metric on $M$.  But Christoffel symbols can be expressed in terms of components of Riemannian metric tensor and their derivatives, so the  Christoffel symbols $\widetilde{\Gamma}^n_{m,\ell}$ of the manifold $M_{\infty}^{(y_{k;i})}$  are limit values in $C^{\infty}$ of the   Christoffel symbols $\Gamma^n_{m,\ell}$  of the manifold M. Therefore $t\mapsto t v$, $0\le t\le\rho$, are geodesic curves also for $M_{\infty}^{(y_{k;i})}$ in the coordinates $(\Omega_\rho,\varphi_i)$.
Thus  the injectivity  radius of $M_\infty^{(y_{k:i})}$ is not smaller then $\rho$  and  $(\Omega_\rho, \varphi_i)$ is a normal system of coordinates. 
   
In terms of the definition above the argument of this subsection proves
the following statement. 
\begin{prop}
Let $M$ be a Riemannian  manifold with bounded geometry and let $Y$ be its $\hat\rho$-discretization, $\frac{\rho}{2}<\hat{\rho}<\rho < \frac{r(M)}{8}$. Then for every discrete sequence
$(y_{k})$ in $Y$ and its trailing system $\lbrace(y_{k;i})_{k\in\N}\rbrace_{i\in\N_{0}}$ there exists a renamed subsequence $(y_{k})$  that generates a Riemannian manifold at infinity $M_{\infty}^{(y_{k;i})}$ of the manifold $M$. The manifold $M_{\infty}^{(y_{k;i})}$ has bounded geometry and its injectivity radius is greater or equal than $\rho$. 
\end{prop}

\begin{rem}\label{rem:bg}
%
Let $M'$ be another manifold such that $M$ and $M'$ have respective
compact subsets $M_{0}$ and $M_{0}'$ such that $M\setminus M_{0}$
is isometric to $M'\setminus M_{0}'$, i. e. let $M'$ and $M$ coincide up to
a compact perturbation. Then  their respective manifolds at infinity for the same trailing systems coincide.
From this follows that  manifold at infinity of the manifold $M$  is not necessarily diffeomorphic 
to $M$. 

\end{rem}

\section{Local and global profiles. Formulation of the main result}

In this section we state our main result. We will use the notation introduced in the last section. In particular we  will work with discrete sequences of points and related trailing systems described in Definition   \ref{def:trailing}. 

\begin{defn}
Let $M$ be a manifold of bounded geometry and $Y$ be its discretization. Let $(u_{k})$ be a bounded
sequence in $H^{1,2}(M).$ Let $(y_{k})$ be a sequence of points
in $Y$
and let $\lbrace(y_{k;i})_{k\in\N}\rbrace_{i\in\N_{0}}$
be its trailing system. 
One says that $w_{i}\in H^{1,2}(\Omega_{\rho})$ is  \emph{a local profile}
of $(u_{k})$ relative to a trailing sequence $(y_{k;i})_{k\in\N}$, if,
on a renamed subsequence, $u_{k}\circ e_{y_{k;i}}\rightharpoonup w_{i}$
 in $H^{1,2}(\Omega_{\rho})$ as $k\to\infty$. If $(y_{k})$ is a
renamed (diagonal) subsequence such that $u_{k}\circ e_{y_{k;i}}\rightharpoonup w_{i}$
in $H^{1,2}(\Omega_{\rho})$ as $k\to\infty$ for all $i\in\N_{0}$, then
the family $\lbrace w_{i}\rbrace_{i\in\N_{0}}$
is called an \emph{array of local profiles of $(u_{k})$ relative to the trailing system $\lbrace(y_{k;i})_{k\in\N}\rbrace_{i\in\N_{0}}$ of the sequence $(y_{k})$}.
\end{defn}


\begin{prop}
\label{prop:W}Let $M$ be a manifold of bounded geometry and let $Y$ its discretization. Let 
$(u_{k})$ be a bounded sequence in $H^{1,2}(M).$  Let $\lbrace w_{i}\rbrace_{i\in\N_{0}}$
be an array of local profiles of $(u_{k})$ associated with a trailing system $\lbrace(y_{k;i})_{k\in\N}\rbrace_{i\in\N_{0}}$ related to the sequence $(y_k)$ in
$Y$. Then there exists a function $w:\ M_{\infty}^{(y_{k;i})}\to\R$
such that $w\circ\varphi_{i}=w_{i}$, $i\in\N_{0}$, where $\varphi_{i}:\Omega_{\rho}\to M_{\infty}^{(y_{k;i})}$ are local coordinate maps of 
$ M_{\infty}^{(y_{k;i})}$.
\end{prop}

\begin{proof}
Functions $w_{i}$ are defined on $\Omega_\rho$ that is a domain of definition of $\varphi_i$.  
Set $w\eqdef w_{i}\circ\varphi_{i}^{-1}$
on $\varphi_{i}^{-1}(\Omega_{\rho})$ and note that if $x\in\varphi_{i}^{-1}(\Omega_{\rho})\cap\varphi_{j}^{-1}(\Omega_{\rho})$
for some $j\in\N_{0}$, then $\varphi_{i}(x)\in\Omega_{ij}$, $\varphi_{j}(x)\in\Omega_{ji}$,
and, using the a.e. convergence of $u_{k}\circ e_{y_{k;i}}$ and $u_{k}\circ e_{y_{k;j}}$
to $w_{i}$ and $w_{j}$ respectively, and the uniform convergence
of $e_{y_{k;i}}^{-1}e_{y_{k;j}}$ to $\psi_{ij}$, we have
\begin{align}
w_{j}\circ\varphi_{j}^{-1} & =
\lim_{k\to\infty}u_{k}\circ e_{y_{k;j}}\circ\varphi_{j}^{-1}= 
\lim_{k\to\infty}u_{k}\circ e_{y_{k;i}}\circ e_{y_{k;i}}^{-1}\circ e_{y_{k;j}}\circ\varphi_{j}^{-1}=\nonumber
\\ \nonumber
& = w_{i}\circ\psi_{ij}\circ\varphi_{j}^{-1}=w_{i}\circ\varphi_{i}^{-1}\circ\varphi_{j}\circ\varphi_{j}^{-1}=w_{i}\circ\varphi_{i}^{-1}
\end{align}
almost everywhere in $\varphi_{i}^{-1}(\Omega_{\rho})\cap\varphi_{j}^{-1}(\Omega_{\rho})$. 
\end{proof}

\begin{defn}\label{def:gp}
Let $\lbrace w_{i}\rbrace_{i\in\N_{0}}$ be a local profile array of a bounded sequence $(u_{k})$
in $H^{1,2}(M)$ relative to a trailing system  $\lbrace(y_{k;i})_{k\in\N}\rbrace_{i\in\N_{0}}$. The function
$w:M_{\infty}^{(y_{k;i})}\to\R$ given by Proposition \ref{prop:W}
is called the \emph{global profile} of the sequence $(u_{k})$ relative
to $(y_{k;i})$.
\end{defn}

Let us fix a smooth partition of unity $\lbrace\chi_{y}\rbrace_{y\in Y}$ subordinated to the uniformly finite covering of $M$ by geodesic
balls $\lbrace B(y,\rho)\rbrace_{y\in Y}$, given by Lemma~\ref{lem:PU}.

\begin{defn}
Let $M$ be a manifold of bounded geometry and let $Y$ be its discretization. Let $M_{\infty}^{(y_{k;i})}$   be a manifold at infinity of $M$ generated by a trailing system  $\lbrace(y_{k;i})_{k\in\N}\rbrace_{i\in\N_{0}}$. 
An \emph{elementary concentration} associated with a function $w:M_{\infty}^{(y_{k;i})}\to\R$ is a sequence $(W_{k})_{k\in\N}$ of functions $M\to\R$ given
by
\begin{equation}
W_{k}=\sum_{i\in\N_{0}}\chi_{y_{k;i}}w\circ\varphi_{i}\circ e_{y_{k;i}}^{-1},\qquad k\in\N.\label{eq:elem-conc}
\end{equation}
where   $\varphi_{i}$ are the local coordinate maps of manifold $M_{\infty}^{(y_{k;i})}$.
\end{defn}

In heuristic terms, after we find limits $w_{i}$, $i\in\N_{0},$ of  the sequence $(u_{k})$ under the ``trailing spotlights'' $(e_{y_{k;i}})_{k\in\N_{0}}$
that follow different trailing sequences $(y_{k;i})_{k\in\N}$ of $(y_{k})$, we give an approximate reconstruction $W_{k}$ of $u_{k}$
``centered'' on the moving center $y_{k}$ of the ``core spotlight''.
We do that by first splitting $w$ into local profiles $w\circ\varphi_{i}$, $i\in\N_0$, on the set $\Omega_\rho$, casting them onto the manifold $M$ in the vicinity of $y_{k;i}$ by composition with $e_{y_{k;i}}^{-1}$,
and patching all such compositions together by the partition of unity
on $M$. Such reconstruction approximates $u_{k}$ on geodesic balls
$B(y_{k},R)$ with any $R>0$, but it ignores the values of $u_{k}$
for $k$ large on the balls $B(y_{k}',R)$, with $d(y_{k},y_{k}')\to\infty$,
where $u_{k}$ is approximated by a different local concentration.
It has been shown in \cite{FiesTin} for the case of manifold $M$ with  cocompact action of a group of isometries
(in particular, for homogeneous spaces) that a global reconstruction
of $u_{k}$, up to a remainder vanishing in $L^p(M)$, is a sum elementary concentrations associated with all such mutually decoupled sequences. 

Similarly, the profile decomposition theorem below, which is the main result
of this paper, says any bounded sequence $(u_{k})$ in $H^{1,2}(M)$
has a subsequence that, up to a remainder vanishing in $L^{p}(M)$, $p\in(2,2^*)$, equals a sum of decoupled elementary concentrations. 

In the theorem and next sections we will work with countable families of discrete sequences of the set $Y$. To each sequence we assign a trailing system so in consequence also a the manifold at infinity. To simplify the notation we will index the  sequences in $Y$, the related trailing systems   the corresponding manifolds,  concentration profiles on these manifolds, etc. by $n$, i.e. we will write     $y_k^{(n)}$, $y^{(n)}_{k;i}$,   $M^{(n)}_\infty$, $w^{(n)}$, etc.      

\begin{thm}
\label{thm:main}
Let $M$ be a manifold of bounded geometry and let $Y$ be its discretization.  Let
$(u_{k})$ be a sequence in $H^{1,2}(M)$ weakly convergent to some function $w^{(0)}$ in $H^{1,2}(M)$. 
Then there exists a renamed subsequence of $(u_{k})$, sequences $(y_{k}^{(n)})_{k\in\N}$ in $Y$ 
, and associated with them global profiles $w^{(n)}$ 
on the respective manifolds at infinity  
$M_{\infty}^{(n)}$, 
$n\in\N$, such that $d(y_{k}^{(n)},y_{k}^{(m)})\to\infty$ when $n\neq m$, and 
\begin{equation}
u_{k}-w^{(0)}-\sum_{n\in\N}W_{k}^{(n)}\to 0\mbox{ in }L^{p}(M),\; p\in(2,2^{*}),\label{eq:PD}
\end{equation}
where $W_{k}^{(n)}=\sum_{i\in\N_{0}}\chi_{i}^{(n)}w^{(n)}\circ\varphi_{i}^{(n)}\circ e_{y_{k;i}^{(n)}}^{-1}$ are elementary concentrations,  $\varphi_{i}^{(n)}$ are the local coordinates of the manifolds $M^{(n)}_\infty$ and $\{\chi_i^{(n)}\}_{i\in\N_0}$ are the corresponding partitions of unity satisfying \eqref{eq:dalpha}. The series $\sum_{n\in\N}W_{k}^{(n)}$ converges in $H^{1,2}(M)$
unconditionally and uniformly in $k\in\N$. Moreover,
\begin{equation}
\|w^{(0)}\|_{H^{1,2}(M)}^{2} +\sum_{n=1}^{\infty}\|w^{(n)}\|_{H^{1,2}(M^{(n)}_\infty)}^{2}\le  \limsup\|u_{k}\|_{H^{1,2}(M)}^{2}\ ,\label{eq:Plancherel}
\end{equation}
 and 
\begin{equation}
\int_{M}|u_{k}|^{p}d\dv\to\int_{M}|w^{(0)}|^{p}\dv+\sum_{n=1}^{\infty}\int_{M_{\infty}^{(n)}}|w^{(n)}|^{p}\dvn.
\label{eq:newBL}
\end{equation}
\end{thm}

\section{Auxiliary statements concerning profile decomposition}

In Sections 5, 6 and 7 we assume that conditions of Theorem \ref{thm:main}
hold true. First we prove the inequality for the norms introduced in Lemma \ref{eqv:norms}.
\begin{lem} 
\label{lem:Plancherel0}Let $(u_{k})$ be a bounded sequence in $H^{1,2}(M)$,
let $M_{\infty}^{(y_{k;i})}$ be a manifold at infinity of $M$ generated
by a trailing system $\lbrace(y_{k;i})_{k\in\N}\rbrace_{i\in\N_{0}}$, and let $w\in H^{1,2}(M_{\infty}^{(y_{k;i})})$
be the associated global profile of $(u_{k})$. Then 
\[
\liminf\nl u_{k}\nr_{H^{1,2}(M)}^{2}\ge\nl w\nr_{H^{1,2}(M_{\infty}^{(y_{k;i})})}^{2}
\]
\end{lem}
\begin{proof}

Let $\{\chi_y\}_{y\in Y}$ be the partition of unity given by Lemma~\ref{lem:PU}, and let us enumerate it for each $k\in\N$ according to the enumeration $\{y_{k;i}\}_{i\in\N_0}$ of $Y$, namely $i\mapsto \chi_{y_{k;i}}$, $i\in\N_0$.  In other words, for every $k$ the set $\{ \chi_{y_{k;i}}\}_{i_\in\N_0}$ equals the set $\{\chi_y\}_{y\in Y}$, and only its enumeration depends on the given trailing system  $\lbrace(y_{k;i})_{k\in\N}\rbrace_{i\in\N_{0}}$.
By Ascoli-Arzela theorem,   we can define  for any $i$ a function  $\eta_i$  on $\Omega_\rho$ by the formula  
\begin{equation}\label{eq:eta}
\eta_{i} = \lim_{k\rightarrow \infty} \chi_{y_{k;i}}\circ e_{y_{k;i}}
\end{equation}
The functions $\eta_i$ are smooth functions supported in $\Omega_\rho$.  Moreover, using  the diagonalization argument if needed, we get 
\[
\eta_{i} = 
\lim_{k\rightarrow \infty} \chi_{y_{k;i}}\circ e_{y_{k;j}}\circ e_{y_{k;j}}^{-1}\circ e_{y_{k;i}}=\eta_{j}\circ\psi_{ji}.
\] 
Since $\sum_{i\in\N_{0}}{\chi}_{y_{k;i}}\circ e_{y_{k;j}}=1$ on $\Omega_{\rho}$
for any $j\in\N_{0}$, we have in the limit $\sum_{i\in\N_{0}:\,(i,j)\in\mathbb{K}}\eta_{i}\circ\psi_{ij}=1$
on $\Omega_{\rho}$, cf. Lemma \ref{Ji}. So the family of the functions
\begin{equation}\label{eq:chii}
\chi^{(y_{k;i})}_i\eqdef \eta_i\circ\varphi_{i}^{-1},\qquad i\in\N_{0}
\end{equation}
is a partition of unity on $M_{\infty}^{(y_{k;i})}$, subordinated
to the covering $\lbrace\varphi_{i}(\Omega_{\rho})\rbrace_{i\in\N_{0}}$ of $M_{\infty}^{(y_{k;i})}$, and it is easy to see that it satisfies \eqref{eq:dalpha}.
 
Both the manifolds $M$ and $M_{\infty}^{(y_{k;i})}$ have bounded geometry, and therefore
\begin{align}
 \liminf_{k\rightarrow \infty}  &\nl u_{k}\nr_{H^{1,2}(M)}^2 = \liminf_{k\rightarrow \infty}\sum_{i\in\N_{0}} \| \big(\chi_{y_{k;i}} u_{k}\big)\circ e_{y_{k;i}}\|^2_{H^{1,2}(\R^N)}\ge \\
  \ge  & \, \sum_{i\in\N_{0}}  \liminf_{k\rightarrow \infty} \|\big(\chi_{y_{k;i}} u_{k}\big)\circ e_{y_{k;i}}\|^2_{H^{1,2}(\R^N)} \ge 
   \sum_{i\in\N_{0}}  \ \|{\eta}_{i} w_i\|^2_{H^{1,2}(\R^N)} = \nonumber \\
 \,= &\, \sum_{i\in\N_{0}}  \ \|{\chi}^{(y_{k;i})}_{i} w\circ \varphi_i\|^2_{H^{1,2}(\R^N)} \ge  \nl w\nr_{H^{1,2}(M_{\infty}^{(y_{k;i})})}^{2}\nonumber 
\end{align}
\end{proof}
\begin{lem}
\label{lem:decoupling}  Let $\lbrace(y_{k;i})_{k\in\N}\rbrace_{i\in\N_{0}}$ be a trailing system for a discrete sequence $(y_{k})$  and let $w\in H^{1,2}(M_{\infty}^{(y_{k;i})})$.  Then  the elementary concentration $W_{k}^{(y_{k;i})}$   associated with this  system  
belongs to $H^{1,2}(M))$. Moreover there is a positive constant $C$ independent of $k$ and $i$ such that   
\begin{equation}\label{eq:Wkw}
\|W_{k}^{(y_{k;i})} \|_{H^{1,2}(M)} \le C\, \|w\|_{H^{1,2}(M_{\infty}^{(y_{k;i})})} 
\end{equation}

If   $(y'_k)_{k\in\N}$ is a  discrete sequence  such that $d(y_{k},y_{k}')\to\infty$, then the elementary concentration $W_{k}^{(y_{k;i})}$
satisfies
\[
W_{k}^{(y_{k;i})}\circ e_{y'_{k}}\to0
\]
in $H^{1,2}(\Omega_{\rho})$.
\end{lem}
\begin{proof}
We recall that  
\begin{equation}
W_{k}^{(y_{k;i})}=\sum_{i\in\N_{0}} \chi_{y_{k;i}}\, w\circ\varphi_{i}\circ e_{y_{k;i}}^{-1}, \label{eq:Wkdef}
\end{equation}
cf. \eqref{eq:elem-conc}. The functions $ \chi_{y_{k;i}} \circ e_{y_{k;i}}$ are  smooth compactly supported functions on $\Omega_\rho$ and the family $\big\{\chi_{y_{k;i}} \circ e_{y_{k;i}}\big\}$ is a bounded set in $C^\infty(\Omega_\rho)$. By the boundedness of the geometry, cf. Lemma \ref{Ji} and Lemma \ref{eqv:norms},  and using \eqref{eq:chii}, we have 
\begin{align} \nonumber
& \|\chi_{y_{k;i}} \circ e_{y_{k;i}}\, w\circ \varphi_i\|^2_{H^{1,2}(\R^N)}  \le   C\|\chi_{y_{k;i}} \circ e_{y_{k;i}}\circ \tau_i w\|^2_{H^{1,2}(M_{\infty}^{(y_{k;i})})} \le  \\
&\qquad\qquad\qquad\le C \sum_{j:\,(i,j)\in \mathbb K}\|\chi^{(y_{k;i})}_i w\|^2_{H^{1,2}(M_{\infty}^{(y_{k;i})})} \nonumber
\end{align} 
So using once more Lemma \ref{eqv:norms} we get 
    \begin{align}\label{eq:Wkw1}
\|W_{k}^{(y_{k;i})} \|^2_{H^{1,2}(M)} \le  &\, C \sum_{i}\|\chi_{y_{k;i}}\circ e_{y_{k;i}}\, w\circ\varphi_{i} \|^2_{H^{1,2}(\R^{^N})} \le\\  
 \le &\,	C \sum_{i} \sum_{j:\,(i,j)\in \mathbb K}\|\chi^{(y_{k;i})}_j w\|^2_{H^{1,2}(M_{\infty}^{(y_{k;i})})} \le C \|w\|^2_{H^{1,2}(M_{\infty}^{(y_{k;i})})} . \nonumber
	\end{align}
This proves \eqref{eq:Wkw}. 

Let $\epsilon>0$. If follows from \eqref{eq:Wkw1} that there exist  $N_{\epsilon}\in\N$ independent of $k$ such that  
\begin{equation}\label{eq:Wk1}
\sum_{i\ge N_{\epsilon}} \|\chi_{y_{k;i}}\circ e_{y_{k;i}}\, w\circ\varphi_{i} \|^2_{H^{1,2}(\R^{^N})} \le \epsilon
\end{equation}
By \eqref{eq:Wkdef} we have 
\begin{equation}
W_{k}^{(y_{k;i})}\circ e_{y'_{k}}=\sum_{i\in I_k}(\chi_{y_{k;i}}w\circ\varphi_{i}\circ e_{y_{k;i}}^{-1})\circ e_{y'_{k}}, \label{eq:Wk}
\end{equation}
where $I_k =\{i: B(y_{k}',\rho)\cap B(y_{k;i},\rho)\not= \emptyset \}$.  
Since $d(y_{k},y'_{k})\to\infty$, we have 
$$\sup_{i\le N_{\epsilon}}d(y_{k;i},y_{k}')\ge d(y_{k},y_{k}')-2N_{\epsilon}\rho\to\infty$$
as $k\to\infty$, and thus $B(y_{k}',\rho)\cap B(y_{k;i},\rho)=\emptyset$
for all $i\le N_{\epsilon}$ if $k$ is sufficiently large. Then $\sum_{i=1}^{N_{\epsilon}}(\chi_{y_{k;i}}w\circ\varphi_{i})\circ e_{y_{k;i}}^{-1}\circ e_{y_{k}'}=0$
for all $k$ large, which together with (\ref{eq:Wk1}) proves the
lemma. 
\end{proof}
\begin{lem}
\label{lem:projection} Let $w$ be a profile of the sequence $(u_k)$, given by Proposition~\ref{prop:W} relative to a trailing system $\{(y_{k;i})_{k\in\N}\}_{i\in\N_0}$, and let $W_k$ be the associated concentration sequence. The following holds true:
\begin{equation}
\lim_{k\to\infty}\langle u_{k},W_{k}\rangle_{H^{1,2}(M)}=\|w\|^2_{H^{1,2}(M_\infty^{(y_{k;i})})}.\label{eq:P2}
\end{equation}
\end{lem}
\begin{proof}
We use for each $k\in\N$ an enumeration of the covering $\lbrace B(y,\rho)\rbrace_{y\in Y}$ by the points $y_{k;i}$ from the trailing system $\lbrace(y_{k;i})_{k\in\N}\rbrace_{i\in\N_{0}}$. 
Taking into account that, as $k\to\infty$, $u_{k}\circ y_{k;j}\rightharpoonup w_{j}$,
$e_{y_{k;i}}^{-1}\circ e_{y_{k;j}}\to\psi_{ij}$,
and $w_{i}\circ\psi_{ij}=w_{j}$, and using the expression $o^w(1)$ for any sequence of functions that converges weakly to zero in $H^{1,2}(\Omega_\rho)$, we have 
\begin{align}\label{eq:wnk}
\langle u_{k}, W_{k}&\rangle_{H^{1,2}(M)} =   \sum_{j\in \N_0} \int_{B(y_{k;j},\rho)} \chi_{y_{k;j}}(x)  u_k(x)  W_k(x)\dvx + \\
 +& \sum_{j\in \N_0} \int_{B(y_{k;j},\rho)} \chi_{y_{k;j}}(x) g\big(\nabla u_k(x), \nabla W_k(x)\big) \dvx ,  \nonumber
\end{align}
and
\begin{align}\label{eq:wn}
\|w\|&^2_{H^{1,2}(M_{\infty}^{(y_{k;i})})} =  \sum_{j\in \N_0}\int_{B(y_{k;j},\rho)} \chi_{j}^{(y_{k;j})}(x) | w(x)|^2 \dtvnx + \\  
+ & \sum_{j\in \N_0} \int_{B(y_{k;j},\rho)}{\chi}_{j}^{(y_{k;i})}(x) g\big(\nabla w(x), \nabla w(x)\big) \dtvnx ,\nonumber
\end{align}
where the functions ${\chi}_{j}^{(y_{k;i})}$ are defined by the formulae \eqref{eq:eta}-\eqref{eq:chii} relative to the trailing system $\{(y_{k;i})_{k\in\N}\}_{i\in\N_0}$. 

Both coverings are uniformly locally finite, so it is sufficient to prove  local identities 
\begin{align} \label{eq:uWw1}
& \lim_{k\to\infty} \int_{B(y_{k;j},\rho)} \chi_{y_{k;j}}(x)  u_k(x)  W_k(x)\dvx =  \\
&\qquad\qquad \qquad \qquad   
\int_{B(y_{k;j},\rho)} \chi_{j}^{(y_{k;i})}(x) | w(x)|^2\dtvnx \nonumber
\end{align}
and
\begin{align}
& \lim_{k\to\infty}  \int_{B(y_{k;j},\rho)} \chi_{y_{k;j}}(x) g\big(\nabla u_k(x), \nabla W_k(x)\big) \dvx = \label{eq:uWw2} \\
&\qquad\qquad \qquad  \int_{B(y_{k;j},\rho)} \chi_{j}^{(y_{k;i})}(x) g\big(\nabla w(x), \nabla w(x)\big) \dtvnx , \nonumber 
\end{align}
In the first case we have  
\begin{align*}
\int_{\Omega_{\rho}}&\chi_{y_{k;j}}\circ e_{y_{k;j}}(\xi) u_k\circ e_{y_{k;j}}(\xi)\, \times \\
&\qquad\qquad \times 
\sum_{i\in\N_{0}}[\chi_{y_{k;i}} \;w\circ \varphi_{i}\circ e_{y_{k;i}}^{-1})]\circ e_{y_{k;j}}(\xi)
\sqrt{g(\xi)}\;\mathrm{d}\xi=
\\
& \int_{\Omega_{\rho}}\chi_{y_{k;j}}\circ e_{y_{k;j}}(\xi)
(w_{j}+o^{w}(1))(\xi)\,\times \\
&\qquad\qquad  \times \sum_{i\in\N_{0}} \chi_{y_{k;i}}\circ e_{y_{k;j}}\;w_{i}\circ(\psi_{ij}+o^w(1))(\xi)
\sqrt{g(\xi)}\;\mathrm{d}\xi= 
\\
\int_{\Omega_{\rho}} & \chi_{y_{k;j}}\circ e_{y_{k;j}}(\xi)
(w_{j}+o^{w}(1))(\xi)\,
(w_{j}+o^{w}(1))(\xi) \times \\
&\sqrt{(\widetilde{g}+o^{w}(1))(\xi)}\;\mathrm{d}\xi\longrightarrow  \int_{\Omega_\rho} \chi_{j}^{(y_{k;j})}\circ \varphi_{j}(\xi)|w_j|^2\; \sqrt{\widetilde{g}(\xi)}\;\mathrm{d}\xi \, ,
\end{align*}
where the last inequality follows from the identity $\sum_{i\in\N_{0}} \chi_{y_{k;i}}\circ e_{y_{k;j}}=1$ on $\Omega_\rho$, cf. Lemma \ref{Ji}. This proves \eqref{eq:uWw1}. 

To prove \eqref{eq:uWw2} we first note that 
\begin{align*}  
&\sum_{\nu,\mu=1}^N g^{\nu,\mu}(\xi) \partial_\nu(u_k\circ e_{y_{k;j}})(\xi)\partial_\mu (W_k\circ e_{y_{k;j}})(\xi) = \\
& = \sum_{\nu,\mu=1}^N g^{\nu,\mu}(\xi) \partial_\nu(u_k\circ e_{y_{k;j}})(\xi)  \times \\
&\qquad\qquad\times \partial_\mu\Big(\sum_{i\in\N_{0}}[\chi_{y_{k;i}}
\;w\circ \varphi_{i}\circ e_{y_{k;i}}^{-1})]\circ e_{y_{k;j}}\Big)(\xi)= \\
&\sum_{\nu,\mu=1}^N g^{\nu,\mu}(\xi) \partial_\nu\Big((w_{j}+o^{w}(1))\circ e_{y_{k;j}}\Big)(\xi) \times \\
&\qquad\qquad\times \partial_\mu\Big(
\chi_{y_{k;i}}\circ e_{y_{k;j}}(\xi)\;
w_{i}\circ(\psi_{ij}+o^w(1))\Big)(\xi) = \\
&\sum_{\nu,\mu=1}^N g^{\nu,\mu}(\xi) \partial_\nu\Big((w_{j}+o^{w}(1))\circ e_{y_{k;j}}\Big)(\xi) 
\partial_\mu\Big(w_{j}+o(1))\Big)(\xi). 
\end{align*}
 In consequence 
 \begin{align*}
\int_{\Omega_{\rho}}\chi_{y_{k;j}}\circ e_{y_{k;j}}(\xi) \, 
\sum_{\nu,\mu=1}^N g^{\nu,\mu}(\xi) \partial_\nu(u_k\circ e_{y_{k;j}})(\xi)\partial_\mu (W_k\circ e_{y_{k;j}})(\xi) \sqrt{g(\xi)}\;\mathrm{d}\xi= 
\end{align*}
 \begin{align*}
	= & \int_{\Omega_{\rho}}\chi_{y_{k;j}}\circ e_{y_{k;j}}(\xi) \,
	\sum_{\nu,\mu=1}^N g^{\nu,\mu}(\xi) \partial_\nu\Big((w_{j}+o^{w}(1))\circ e_{y_{k;j}}\Big)(\xi) 
	\\
&\qquad\qquad \partial_\mu\Big((w_{j}+o^{w}(1))\circ e_{y_{k;j}}\Big)(\xi)	\sqrt{\widetilde{g}(\xi)+o(1)}\;\mathrm{d}\xi \longrightarrow 
\\
& 	\int_{\Omega_{\rho}}\chi_{j}^{(y_{k;i})}\circ \varphi_j(\xi) \sum_{\nu,\mu=1}^N \widetilde{g}^{\nu,\mu}(\xi) \partial_\nu w\circ \varphi_j(\xi)\partial_\mu w\circ \varphi_j(\xi) \sqrt{\widetilde{g}(\xi)}\;\mathrm{d}\xi 
 	\end{align*}
Combining the last  calculations with \eqref{eq:wnk}-\eqref{eq:uWw2} we arrive at \eqref{eq:P2}.
\end{proof}
\begin{lem}
\label{lem:length} Let $w$ be a profile of the sequence $u_k$, given by Proposition~\ref{prop:W} relative to a trailing system $\{(y_{k;i})_{k\in\N}\}_{i\in\N_0}$, and let $W_k$ be the associated concentration sequence. The following holds true:
\begin{equation}
\lim_{k\to\infty}\|W_{k}\|^2_{H^{1,2}(M)}=\|w\|_{H^{1,2}(M_{\infty}^{(y_{i;k})})}^{2}.\label{eq:P3}
\end{equation}
\end{lem}
\begin{proof}
We can proceed in the similar way as in the proof of  Lemma \ref{lem:projection}. Once more we can reduce the argumentation  to the local identities using \eqref{eq:wn} and 

\begin{align}
&\|W_{k}\|^2_{H^{1,2}(M)} =  \sum_{j\in \N_0} \int_{B(y_{k;j},\rho)} \chi_{y_{k;j}}(x)  |W_k(x)|^2\dvx + \\
 &\qquad\qquad + \sum_{j\in \N_0} \int_{B(y_{k;j},\rho)} \chi_{y_{k;j}}(x) g\big(\nabla W_k(x) \nabla W_k(x)\big) \dvx ,  \nonumber
\end{align}
We have 
\begin{align*}
	&\int_{\Omega_{\rho}}\chi_{y_{k;j}}\circ e_{y_{k;j}}(\xi) \, 
	\sum_{\nu,\mu=1}^N g^{\nu,\mu}(\xi) \partial_\nu \big(W_k\circ e_{y_{k;j}}\big)(\xi)\partial_\mu \big(W_k\circ e_{y_{k;j}}\big)(\xi) \sqrt{g(\xi)}\;\mathrm{d}\xi= 
	\\
& \qquad 	\int_{\Omega_{\rho}}\chi_{y_{k;j}}\circ e_{y_{k;j}}(\xi) \,
	\sum_{\nu,\mu=1}^N g^{\nu,\mu}(\xi) \partial_\nu\Big((w_{j}+o^{w}(1))\Big)(\xi)\times \\ 
&	\qquad \qquad\qquad\qquad\times \partial_\mu\Big((w_{j}+o^{w}(1))\Big) \sqrt{\widetilde{g}(\xi)+o(1)}\;\mathrm{d}\xi \longrightarrow 
	\\
&\qquad	\int_{\Omega_{\rho}}\chi_{j}^{(y_{k;i})}\circ \varphi_j(\xi) \sum_{\nu,\mu=1}^N \widetilde{g}^{\nu,\mu}(\xi) \partial_\nu \big(w\circ \varphi_j\big)(\xi)\partial_\mu \big(w\circ \varphi_j\big)(\xi) \sqrt{\widetilde{g}(\xi)}\;\mathrm{d}\xi 
\end{align*}

Also as above,
\begin{align*}
\int_{\Omega_{\rho}}\chi_{y_{k;j}}\circ e_{y_{k;j}}(\xi) \; 
\big|\sum_{i\in\N_{0}}[\chi_{y_{k;i}}
\;w\circ \varphi_{i}\circ e_{y_{k;i}}^{-1})]\circ e_{y_{k;j}}(\xi)\big|^2
\sqrt{g}(\xi)\;\mathrm{d}\xi =
\\
\int_{\Omega_{\rho}}\chi_{y_{k;j}}\circ e_{y_{k;j}}(\xi)
\big|\sum_{i\in\N_{0}} 
\chi_{y_{k;i}}\circ e_{y_{k;j}}(\xi)\;
w_{i}\circ(\psi_{ij}+o^w(1))(\xi)\big|\,
\sqrt{g}(\xi)\;\mathrm{d}\xi=
\\
\int_{\Omega_{\rho}}\chi_{y_{k;j}}\circ e_{y_{k;j}}(\xi)
|(w_{j}+o^w(1))(\xi)|^2\,
\sqrt{(\widetilde{g}+o^w(1))(\xi)}\;\mathrm{d}\xi\quad\longrightarrow 
\\
\int_{\Omega_{\rho}}\chi_j^{(y_{k;i})}\circ \varphi_{j}(\xi)
|(w_{j}(\xi)|^2\,
\sqrt{\widetilde{g}}(\xi)\;\mathrm{d}\xi
\end{align*}
\end{proof}
Below we consider a countable family of trailing systems $\{(y_{k;j}^{(n)})_{k\in\N}\}_{i\in\N_0}$, $n\in\N$, 
and will abbreviate the notation of the associated manifolds at infinity,
$M^{(y_{k;j}^{(n)})}_{\infty}$, as $M^{(n)}_\infty$. This convention will also extend to all other objects generated by trailing systems $\{(y_{k;i}^{(n)})_{k\in\N}\}_{i\in\N_0}$, but not to objects  
 indexed by points in $Y$, such as $\chi_{y_{k;i}^{(n)}}$.

\begin{lem}
\label{lem:Plancherel}Assume that $u_{k}\rightharpoonup0$. Assume that trailing systems $\{(y_{k;i}^{(n)})_{k\in\N}\}_{i\in\N_0}$ of discrete sequences $(y_k^{(n)})_{k\in\N}$, $n\in\N$,  generate local profiles $\lbrace w_{i}^{(n)}\rbrace_{i\in\N_{0}}$,
such that $d(y_{k}^{(n)},y_{k}^{(\ell)})\to\infty$ when $n\neq\ell$.
Then 
\begin{equation}
\sum_{n=1}^{m}\|w^{(n)}\|_{H^{1,2}(M_{\infty}^{(n)})}^{2}\le\limsup\|u_{k}\|_{H^{1,2}(M)}^{2}.\label{eq:P0}
\end{equation}
\end{lem}
\begin{proof}
Consider for each $n=1,\dots,m$ the elementary concentrations $W_{k}^{(n)}=\sum_{i\in\N_{0}}\chi_{y_{k;i}^{(n)}}w_{i}^{(n)}\circ e_{y_{k;i}^{(n)}}^{-1}$,
$w_{i}^{(n)}=w^{(n)}\circ\varphi_{i}^{(n)}$, where $\lbrace\varphi_{i},\Omega_{\rho}\rbrace_{i\in\N_{0}}$
is the atlas of the manifold at infinity $M_{\infty}^{{n}}\eqdef M_{\infty}^{(y_{k;i}^{(n)})}$,
and let us expand by bilinearity the trivial inequality
\[
\left\Vert u_{k}-\sum_{n=1}^{m}W_{k}^{(n)}\right\Vert _{H^{1,2}(M)}^{2}\ge0.
\]
For convenience, the subscript in the Sobolev norm will be omitted
for the rest of this proof. We have then

\begin{equation}
2\sum_{n=1}^{m}\langle u_{k},W_{k}^{(n)}\rangle-\sum_{n=1}^{m}\|W_{k}^{(n)}\|^{2}\le\|u_{k}\|^{2}+\sum_{n\neq\ell}\langle W_{k}^{(n)},W_{k}^{(\ell)}\rangle.\label{eq:P1-1}
\end{equation}
Applying Lemmas \ref{lem:projection} and \ref{lem:length} we have
\begin{equation}
\sum_{n=1}^{m}\|w^{(n)}\|_{H^{1,2}(M_{\infty}^{(n)})}^{2}\le\|u_{k}\|^{2}+\sum_{n\neq\ell}\langle W_{k}^{(n)},W_{k}^{(\ell)}\rangle+o(1).\label{eq:P1}
\end{equation}
In order to prove the lemma it suffices therefore to show that $\langle W_{k}^{(n)},W_{k}^{(\ell)}\rangle\to0$
whenever $n\neq\ell$. 

Since $d(y_{k}^{(n)},y_{k}^{(\ell)})\to\infty$, we also have
$d(y_{k;i}^{(n)},y_{k;j}^{(\ell)})\to\infty$ for any $i,j\in\N_{0}$.
Let $\epsilon>0$ and let $N_{\epsilon}\in\N$ be such that, in view
of Lemma \ref{lem:Plancherel0}, 
\begin{align}\label{eq:chwost}
\sum_{i\ge N_{\epsilon}} \, &\,\int_{\Omega\rho}\chi_{i}^{(n)}({\xi})
\sum_{\nu,\mu=1}^N  g^{\nu\mu}(\xi)\partial_n(w_{i}^{(n)})(\xi)\partial_\mu(w_{i}^{(n)})(\xi))+\\
&\qquad\qquad  + |w_{i}^{(n)}(\xi)|^{2}]\sqrt{g(\xi)}\mathrm{d}\xi \le\epsilon,\; n=1,\dots,m.\nonumber
\end{align}
Let $W_{k}^{(n)}=W_{k}^{(n)'}+W_{k}^{(n)''}$ where 
\[
W_{k}^{(n)'}=\sum_{i<N_{\epsilon}}(\chi_{y_{k;i}^{(n)}}w_{i}^{(n)}\circ e_{y_{k;i}^{(n)}}^{-1})\mbox{ and }W_{k}^{(n)''}=\sum_{i\ge N_{\epsilon}}(\chi_{y_{k;i}^{(n)}}w_{i}^{(n)}\circ e_{y_{k;i}^{(n)}}^{-1})
\]
 and note that for all $k$ sufficiently large, $W_{k}^{(n)'}$ and
$W_{k}^{(\ell)'}$ have disjoint supports. Thus 
\begin{equation}
|\langle W_{k}^{(n)},W_{k}^{(\ell)}\rangle|\le 2S_{k}T_{k} + T_{k}^{2},\label{eq:ST}
\end{equation}
where $S_{k}=\max_{n=1,\dots m}\|W_{k}^{(n)'}\|$ and $T_{k}=\max_{n=1,\dots m}\|W_{k}^{(n)''}\|$. The estimate for $S_{k}$ is readily provided by repeating verbally
the argument of Lemma \ref{lem:length}, which gives 
\[
S_{k}^{2}\le\max_{n=1,\dots,m}\|w^{(n)}\|_{H^{1,2}(M_{\infty}^{(n)})}^{2}+o(1),
\]
so $S_{k}$ is bounded by $C \|u_{k}\|+o(1)$ due to Lemma \ref{lem:Plancherel0},
while a similar adaptation of Lemma \ref{lem:length} to summation for
$i\ge N_{\epsilon}$ yields that $T_{k}^{2}$ is bounded, up to vanishing
terms, by the left hand side of (\ref{eq:chwost}), and thus $T_{k}\le\sqrt{\epsilon}+o(1)$.
Thus from (\ref{eq:ST}) we have 
\[
|\langle W_{k}^{(n)},W_{k}^{(\ell)}\rangle|\le C\sqrt{\epsilon}(\|u_{k}\|+\sqrt{\epsilon}+o(1)),
\]
which implies, in turn, that $\limsup_{k\to\infty}|\langle W_{k}^{(n)},W_{k}^{(\ell)}\rangle|\le C\sqrt{\epsilon}$,
and since $\epsilon$ is arbitrary, we have $\langle W_{k}^{(n)},W_{k}^{(\ell)}\rangle\to0$
for $n\neq\ell$, which completes the proof.
\[
\]
\end{proof}
Before we begin the proof of Theorem \ref{thm:main}, we introduce
the following technical definition.
\begin{defn}
\label{def:modulus}Let $(u_{k})$ be a bounded sequence in $H^{1,2}(M)$.
Let $(y_{k}^{(\ell)})$, $\ell=1,\dots  ,m$, $m\in\N$, 
be discrete sequences of points
in $Y$, satisfying $d(y_{k}^{(n)},y_{k}^{(\ell)})\to\infty$ for
$n\neq\ell$, and generating global profiles $w_{1},\dots,w_{m}$
of a renamed subsequence of $(u_{k})$ in respective Sobolev spaces
$H^{1,2}(M_{\infty}^{(\ell)})$. 
A modulus
$\nu^{(u_{k})}((y_{k}^{(1)}),\dots,(y_{k}^{m}))$ of this subsequence
is the supremum of the set of values $\|w\|_{H^{1,2}(M_{\infty}^{(y_{k;i})})}^{2}$
of all global profiles $w$ of the renamed subsequence $(u_{k})$
generated by a trailing system $\{(y_{i;k})_{k\in\N}\}_{i\in\N_0}$ in $Y$ satisfying $d(y_{k;0},y_{k}^{(\ell)})\to\infty$, $\ell=1,\dots,m$.
If such set is empty, we set $\nu^{(u_{k})}((y_{k}^{(1)}),\dots,(y_{k}^{(m)}))\eqdef 0$. For $m=0$,  $\nu^{(u_k)}(\emptyset)$ is defined as the corresponding unconstrained supremum.
\end{defn}


\section{Proof of Theorem \ref{thm:main}.}

\emph{Step 1}. It suffices to prove Theorem \ref{thm:main} for sequences
that weakly converge to zero. Indeed, assume that the theorem is true
in this case. A general bounded sequence $(u_{k})$ in $H^{1,2}(M)$,
it has a renamed subsequence weakly convergent to some $w^{(0)}$
in $H^{1,2}(M)$. Consider then conclusions of the theorem for the
sequence $(u_{k}-w^{(0)})$ . Since for any discrete sequence $(y_{k})$
in $Y$, $w^{(0)}\circ e_{y_{k}}\rightharpoonup0$ in $H^{1,2}(\Omega_{\rho})$ by Lemma~\ref{lem:Plancherel0},
sequences $(u_{k})$ and $(u_{k}-w^{(0)})$ have identical local profiles
under the same trailing systems $\{(y_{i;k}^{(n)})_{k\in\N}\}_{i\in\N_0}$, identical manifolds at
infinity and identical concentration terms $W_{k}^{(n)}$, which yields
\eqref{eq:PD}. Relation (\ref{eq:Plancherel}) follows from the elementary identity for Hilbert space norms,
\[
\|u_{k}\|^{2}-\|w^{(0)}\|^{2}-\|u_{k}-w^{(0)}\|^{2}\to0,
\]
and (\ref{eq:Plancherel}) for the sequence $ $$(u_{k}-w^{(0)})$.
Relation (\ref{eq:newBL}) follows from Brezis-Lieb Lemma (\cite{BL}),
which gives, in our settings, 
\[
\int_{M}|u_{k}|^{p}\mathrm dv_g-\int_{M}|w^{(0)}|^{p}\mathrm dv_g-\int_{M}|u_{k}-w^{(0)}|^{p}\mathrm dv_g\to0,
\]
combined with (\ref{eq:newBL}) for the sequence $(u_{k}-w^{(0)})$. 

From now on we assume that $u_{k}\rightharpoonup0$.

 \emph{Step 2}. Let us give an iterative construction of sequences $(v_{k}^{(n)})_{k\in\N}$
in $H^{1,2}(M)$, $n\in\N_{0}$. We set $v_k^{(0)}=u_k$ and choose 
$(y_{k}^{(1)})_{k\in\N}$ so that $\|w^{(1)}\|_{H^{1,2}(M_\infty^{(1)})}\ge \frac12\nu^{(u_k)}(\emptyset)$. 

Assume that we have defined sequences
$(v_{k}^{(0)})_{k\in\N}$,...,$(v_{k}^{(m)})_{k\in\N}$, with the
following properties:
\begin{quotation}
There exists, for a given $m$, a renamed subsequence of $(u_{k}),$
sequences $(y_{k}^{(1)})_{k\in\N},\dots,(y_{k}^{(m)})_{k\in\N}$ of
points in $Y$ such that $d(y_{k}^{(\ell)},y_{k}^{(n)})\to\infty$
whenever $\ell\neq n$, with trailing systems $\left\lbrace (y_{k;i}^{(n)})_{k\in\N}\right\rbrace _{i\in\N_{0}}$,
defining, on a subsequence, for each respective $n=1,\dots,m$, an
array of local profiles$ $ $\lbrace w_{i}^{(n)}\rbrace_{i\in\N_{0}}$
of  (the $m$th extraction of) $(u_{k})$, and, consequently,
a Riemannian manifold at infinity $M_{\infty}^{(n)}$ and a global
profile $w^{(n)}\in H^{1,2}(M_{\infty}^{(n)})$. Assume, furthermore,
that $\|w^{(n)}\|_{H^{1,2}(M_{\infty}^{(n)})}^{2}\ge\frac{1}{2}\nu^{(u_{k})}((y_{k}^{(1)}),\dots,(y_{k}^{(n-1)}))$,
$n=2,\dots,m$ (cf. Definition \ref{def:modulus}). Let $(W_{k}^{(n)})_{k\in\N}$,
$n=1,\dots,m$, be corresponding elementary concentrations, and define,
with the convention that the sum over an empty set equals zero, 
\[
v_{k}^{(n)}\eqdef u_{k}-\sum_{\ell=1}^{n}W_{k}^{(\ell)},\ n=1,\dots m.
\]

\end{quotation}
Under the above assumptions we construct now a sequence $v_{k}^{(m+1)}$
that will also satisfy these assumptions. Consider all sequences $(y_{k})$ of points in $Y$ such that $d(y_{k},y_{k}^{(\ell)})\to\infty$ for all $\ell=1,\dots,m$. We have three complementary cases:
\begin{itemize}
	\item[case 1:]  for any such sequence one has $v_{k}^{(m)}\circ e_{y_{k}}\rightharpoonup0$
	in $H^{1,2}(\Omega_{\rho})$ on a renamed subsequence;
	\item[case2:]  there
	exists a bounded sequence $(y_{k})$ of points in $Y$ (so that $d(y_{k},y_{k}^{(\ell)})\to\infty$
	for all $\ell=1,\dots,m$) such that, on a renamed subsequence, $v_{k}^{(m)}\circ e_{y_{k}}\rightharpoonup w\neq0$;
	\item[case 3:] there exists a discrete sequence $(y_{k})$ of points
	in $Y$ such that $d(y_{k},y_{k}^{(\ell)})\to\infty$ for all $\ell=1,\dots,m$,
	and $v_{k}^{(m)}\circ e_{y_{k}}\rightharpoonup w\neq0$.
\end{itemize}

Case 2 is in fact vacuous. Indeed, in this case $(y_{k})$ would have
a constant subsequence with some value $z$ and $u_{k}\circ e_{z}\rightharpoonup w\neq0$,
which contradicts the assumption $u_{k}\rightharpoonup0$.

Consider case 1. We prove that in that case $v_{k}^{(m)}\circ e_{z_{k}}\rightharpoonup0$ for any sequence $(z_{k})$ in $Y$. By assumption we know that it is true if $d(z_{k},y_{k}^{(\ell)})\to\infty$ for all $\ell=1,\dots, m$. So let us assume that  on a renamed subsequence, $d(z_{k},y_{k}^{(\ell)})$ is bounded for some $\ell\in\lbrace1,\dots m\rbrace$. Then by the definition of the trailing system there exists $i\in\N_{0}$ such that  $z_{k}=y_{k;i}^{(\ell)}$  on a renamed subsequence. So if $u_{k}\circ e_{z_{k}}\rightharpoonup w\neq0$ then $w$  coincides with the local profile $w_{i}^{(\ell)}$. Moreover $d(z_{k},y_{k}^{(n)})\to\infty$ if $1\le n \le m$ and $n\not=\ell$. So by   Lemma \ref{lem:decoupling}, $W^{(n)}_{k}\circ e_{z_{k}}\rightharpoonup 0$ if $n\not= \ell$ and $W^{(\ell)}_{k}\circ e_{z_{k}}\rightharpoonup w_i$ . In consequence   $v_{k}^{(m)}\circ e_{z_{k}}\rightharpoonup0$   
Now by Lemma \ref{lem:spotlight}, $v_{k}^{(m)}\to0$ in $L^{p}(M)$, which means that the asymptotic relation (\ref{eq:PD}) is proved with a finite sum of elementary concentrations and we can take $v_{k}^{(m+1)}=0$.  

Consider now case 3.  Now the modulus  $\nu^{(u_{k})}((y_{k}^{(1)}),\dots,(y_{k}^{m}))>0$ is positive,  cf. Definition \ref{def:modulus}).
We may choose a sequence $y_{k}^{(m+1)}$,  $d(y^{(m+1)}_{k},y_{k}^{(\ell)})\to\infty$ for all $\ell=1,\dots,m$,
in such a way that the corresponding global profile $w^{(m+1)}$ of $(u_{k})$ satisfies 
\begin{equation}
\|w^{(m+1)}\|_{H^{1,2}(M_{\infty}^{(m+1)})}^{2}\ge\frac{1}{2}\nu^{(u_{k})}((y_{k}^{(1)}),\dots,(y_{k}^{(m)})).\label{eq:choose}
\end{equation}
Then  using the local profiles  $w_{i}^{(m+1)}$, $i\in\N_{0}$, we may  define, for a renamed subsequence, the associated global profile $w^{(m+1)}$ (cf. Proposition~\ref{prop:W}), and the corresponding
elementary concentration
$W_{k}^{(m+1)}$,  and put   
\[
v_{k}^{(m+1)}\eqdef u_{k}-\sum_{\ell=1}^{m+1}W_{k}^{(\ell)}.
\]
It is easy to see that the sequence $(v_{k}^{(m+1)})$ has the same
properties as $(v_{k}^{(n)})$, $n=0,\dots,m$. 

\emph{ Step 3}. By Lemma \ref{lem:Plancherel} we have 
\[\sum_{n=1}^{m}\|w^{(n)}\|_{H^{1,2}(M_{\infty}^{(n)})}^{2}\le\limsup\|u_{k}\|_{H^{1,2}(M)}^{2}
\]
for any $m$, which proves (\ref{eq:Plancherel}).

\emph{Step 4}. In order to prove convergence of the series $\sum_{n=1}^{\infty}W_{k}^{(n)}$
note first that we may assume without loss of generality that for
each $n\in\N$, $ $there exists $r_{n}>0$ such that $\supp W_{k}^{(n)}\subset B(y_{k}^{(n)},r_{n})$.
Indeed, acting like in the proof of Lemma \ref{lem:Plancherel}, from
the calculations in the proof of Lemma \ref{lem:length} one can easily
see that one can approximate $W_{k}^{(n)}$ in the $H^{1,2}$-norm
by restricting summation in (\ref{eq:elem-conc}) to a finite number
of terms, with the norm of the remainder bounded by, say, $\epsilon2^{-n}$
with a small $\epsilon>0$. Then, for any $m\in\N$ one can extract
a subsequence $(k_{j}^{(m)})_{j\in\N}$ of $(k)_{k\in\N}$ such that
$d(y_{k}^{(n)},y_{k}^{(\ell)})>r_{n}+r_{\ell}$ whenever $1\le\ell<n\le m$.
Then on a diagonal subsequence $(k_{m}^{(m)})_{m\in\N}$ the elementary
concentrations $(W_{k}^{(n)})_{k=k_{m}^{(m)}, m\in\N}$ will have pairwise
disjoint supports. Together with (\ref{eq:Plancherel}) this proves
that the convergence is unconditional and uniform with respect to $k$.

\emph{Step 5}.  Now we prove that  $(u_{k}-\sum_{\ell=1}^{\infty}W_{k}^{(\ell)})\circ e_{y_{k}}\to0$
in $L^{p}(M)$ for any sequence $y_k$ in $Y$.

Let first $(y_{k})$ in $Y$ be a bounded sequence. Since
it has finitely many values, on each constant subsequence we have
$u_{k}\circ e_{y}\rightharpoonup0$ and $W_{k}^{(\ell)}\circ e_{y}\rightharpoonup0$,
and thus $(u_{k}-\sum_{\ell=1}^{\infty}W_{k}^{(\ell)})\circ e_{y_{k}}\rightharpoonup0$.

Let now $(y_{k})$ be a discrete sequence in $Y$. If there is $\ell\in \N$ such that on a renamed subsequence we have $d(y_k,y_k^{(\ell)})$ is bounded. Then on a renamed subsequence $y_k=y^{(\ell)}_{k;i}$ for some $i$,  cf. Step 2. But then $u_{k}\circ e_{y_k}\rightharpoonup w^{(\ell)}_i$,  
$W_{k}^{(\ell)}\circ e_{y_k}\rightharpoonup w^{(\ell)}_i$ and $W_{k}^{(n)}\circ e_{y_k}\rightharpoonup 0$ if $n\not= \ell$, cf. Lemma \ref{lem:decoupling}.   Thus $(u_{k}-\sum_{\ell=1}^{\infty}W_{k}^{(\ell)})\circ e_{y_{k}}\rightharpoonup0$.

Let $(y_{k})$ be a discrete sequence in $Y$, such that $d(y_k,y_k^{(\ell)})\rightarrow \infty$ for any  $\ell\in \N_0$. 
Assume that on a renamed subsequence $(u_{k}-\sum_{\ell=1}^{\infty}W_{k}^{(\ell)})\circ e_{y_{k}}\rightharpoonup w_{0}\neq0$.
Then $(y_{k})$ generates a profile $w$ of $(u_{k})$ on some manifold
ar infinity $M_{\infty}$ of $M$ that necessarily satisfies $\|w\|_{H^{1,2}(M_{\infty})}\le\nu^{(u_{k})}((y_{k}^{(1)}),\dots,(y_{k}^{(m)}))$
for any $m\in\N$. By (\ref{eq:Plancherel}) and (\ref{eq:choose})
we have $\nu^{(u_{k})}((y_{k}^{(1)}),\dots,(y_{k}^{(m)}))\to0$ as
$m\to\infty$, and therefore $w=0$, which implies $w_{0}=0$. This gives  the contradiction.

We conclude that $(u_{k}-\sum_{\ell=1}^{\infty}W_{k}^{(\ell)})\circ e_{y_{k}}\rightharpoonup0$
for any sequence $(y_{k})$ in $Y$, and by Lemma \ref{lem:spotlight}
$(u_{k}-\sum_{\ell=1}^{\infty}W_{k}^{(\ell)})\circ e_{y_{k}}\to0$
in $L^{p}(M)$.

\emph{Step 6}. It was proved in Step 4 that the series of elementary concentration $W^{(n)}_k$ is convergent in $H^{1,2}(M)$.  So for any $\epsilon>0$ the sum $S_k$ of the elementary concentrations can be approximated by the finite sum $S^\epsilon_k$ i.e.
\begin{align}
\left|\|u_k\|_p-\|S_k^\epsilon\|_p\right|\le
\left|\|u_k\|_p-\|S_k\|_p\right|+\|S_k-S_k^\epsilon\|_p\le
\\
 o(1)+C\|S_k-S_k^\epsilon\|_{H^{1,2}(M)}\le C\epsilon+o(1). \nonumber
\end{align} 
Moreover similarly to Step 4, we may assume without loss of generality
all $w^{(n)}$ have compact support. In consequence we may assume that  there exists $m\in\N$ such that $w^{(n)}=0$ for all $n>m$,
and that $w^{(n)}$ have compact support if $n\le m$.


Let us now evaluate $\|S_k^\epsilon\|_p$.
Let us show first that 
\begin{equation}\label{eq:X9}
\int_{M}|W_{k}^{(n)}|^{p}\mathrm d v_g\to\int_{M_{\infty}^{(n)}}|w^{(n)}|^{p}\mathrm dv_{\widetilde g^{(n)}}.
\end{equation}
Indeed, omitting for the sake of simplicity the superscript $n$ and
taking into account that $w_{i}\circ e_{y_{k;i}}^{-1}\circ e_{y_{k;j}}\to w_{j}$, $e_{y_{k;j}}^{-1}\circ e_{y_{k;i}}\to\psi_{ji}$, and $\chi_{y_{k;j}}\circ e_{y_{k;j}}\to \chi_j$ as in the proof of Lemma~\ref{lem:Plancherel0}, we have:
\begin{align*}
&\int_{M}|W_{k}|^{p} \dv  =\int_{M}\left|\sum_{i\in\N_{0}}\chi_{y_{k;i}}w_{i}\circ e_{y_{k;i}}\right|^{p}\mathrm d v_g=
\\
 &\qquad =\sum_{j\in\N_{0}}\int_{\Omega_{\rho}}\chi_{y_{k;j}}\circ e_{y_{k;j}}\left|\sum_{i\in\N_{0}}\chi_{y_{k;i}}w_{i}\circ e_{y_{k;i}}\right|^{p}\circ e_{y_{k;j}}^{-1}\sqrt{g_{k;j}}\mathrm d\xi=
\\
&\qquad = \sum_{j\in\N_{0}}\int_{\Omega_{\rho}}(\chi_j+o(1))\left|\sum_{i\in\N_{0}}\chi_{y_{k;i}}\circ e_{y_{k;j}}^{-1}(w_{j}+o(1))\right|^{p}\sqrt{\widetilde{g}_{j}+o(1)}\mathrm d\xi\to
\\
 &\qquad\qquad\qquad\qquad \int_{M_{\infty}}|w|^{p}\mathrm  d v_{\widetilde g}.
\end{align*}
Note that the notation $o(1)$ above refers to functions vanishing in the sense of $C^\infty$ and that all infinite sums contain uniformly finitely many terms. 

Now, for all $k$ sufficiently
large, all elementary concentrations $W_{k}^{(n)}$ in the sum $S_k^\epsilon$ have pairwise disjoint supports, 
and, since $\ell^1\hookrightarrow \ell^\frac{p}{2}$, taking into account \eqref{eq:Plancherel}, we have
\begin{align*}
\left(\sum_{n\ge \nu}\int_{M_{\infty}^{(n)}}|w^{(n)}|^{p}\mathrm dv_{\widetilde g^{(n)}}\right)^\frac{2}{p}\le 
\sum_{n\ge \nu}\left(\int_{M_{\infty}^{(n)}}|w^{(n)}|^{p}\mathrm dv_{\widetilde g^{(n)}}\right)^\frac{2}{p}\le
\\
\sum_{n\ge \nu}
C\|w^{(n)}\|_{H^{1,2}(M^{(n)}_\infty)}^2\to 0 \mbox{ as } \nu\to\infty.
\end{align*}
Therefore (\ref{eq:newBL}) is immediate from (\ref{eq:PD}), which completes the proof of Theorem~\ref{thm:main}. 
\hfill $\Box$

\section{Local and global profile decompositions on cocompact manifolds.}

Let $M$ be now a smooth connected complete Riemannian manifold, cocompact
relative to a subgroup $G$ of its isometry group, that is, we assume that there exists
an open bounded set $\mathcal{O}$ such that $\cup_{g\in G}g\mathcal{O}=M$.
Then $M$ is obviously of bounded geometry. It is then natural
to ask if Theorem \ref{thm:main} yields Theorem \ref{thm:FT} with
the manifolds $M_{\infty}^{(n)}$ isometric to $M$. Below we consider
this question in the case when $G$ is a discrete countable subgroup. Without loss of generality we may assume that $\mathcal{O}$ is a geodesic ball.
\begin{thm}
Let $M$ be a smooth connected $N$-dimensional Riemannian manifold,
let $\rho\in(0,\frac{r(M)}{8})$ and $z\in M$, and assume that there exists
a discrete countable subgroup $G$ of isometries on $M$ such that
$\lbrace B(gz,\rho)\rbrace_{g\in G}$ covers $M$ with a uniformly
finite multiplicity. Then 

(i) one can choose the construction parameters of manifolds $M_{\infty}^{(n)}$,
so that they will coincide, up to isometry, with $M$, 
and

(ii) there exist sequences $(g_{k}^{(n)})_{k\in\N}$, of elements
in $G$, and functions $w^{(n)}\in H^{1,2}(M)$, $n\in\N$, such that
the sequences $([g_{k}^{(\ell)}]^{-1}g_{k}^{(n)})_{k\in\N}$ are discrete
whenever $\ell\neq n$, $u_{k}\circ g_{k}^{(n)}\rightharpoonup w^{(n)}$
in $H^{1,2}(M)$, $n\in\N$, and 
\[
W_{k}^{(n)}=w^{(n)}\circ[g_{k}^{(n)}]^{-1}.
\]
\end{thm}
\begin{proof}
1. Let us repeat the construction of the manifold at infinity relative to a sequence $(y_k)$ in $Y=\{gz\}_{g\in G}$. Fix a sequence $h_{i}\in G$, $h_0=\mathrm{id}$, such that $d(h_{i+1}z,z)\ge d(h_iz,z)$, $i\in\N_{0}$,
and define the $i$th
trailing sequence of $(y_k)$ by
$y_{k;i}\eqdef g_{k}h_{i}z$,  $k\in\N$. 
Recall that the normal coordinates at the points $y\in Y$ were defined as $\exp_y$ up to an arbitrarily fixed isometry on $T_yM$. For the present construction we set them specifically as $e_{gz}\eqdef g\circ e_{z}$. Under such choice the transition maps of $M_{\infty}^{(y_{k;i})}$ are characterized by elements of the group $G$:
\[
\psi_{ij}=\lim_{k\to\infty}e_{y_{k;i}}^{-1}\circ e_{y_{k;j}}=\lim e_{z}^{-1}\circ[g_{k}h_{i}]^{-1}g_{k}h_{j}\circ e_{z}=e_{z}^{-1}\circ h_{i}^{-1}h_{j}\circ e_{z},
\]
 and the sequences above are in fact constant with respect to $k$. Consequently, the transition maps $\psi_{ij}$ of the manifold $M_{\infty}^{(y_{k;i})}$ are $e_{z}^{-1}\circ h_{i}^{-1}h_{j}\circ e_{z}$ - same as of $M$ itself. In other words, all the gluing data for $M_{\infty}^{(y_{k;i})}$ are taken from $M$, which suggests, since Theorem~\ref{thm:gluing} is based on a suitable list of properties of charts of a manifold that will allow its reconstruction that $M_{\infty}^{(y_{k;i})}$ is isometric to $M$. We will, however, apply Corollary~\ref{cor:gluing} formally, as follows. 
 
 Manifold $M_\infty^{(y_{k;i})}$ has an atlas $\lbrace(\varphi_{i}(\Omega_{\rho}),\varphi^{-1}_{i})\rbrace_{i\in\N_{0}}$ with
transition maps $\varphi_{i}^{-1}\varphi_{j}=e_{z}^{-1}\circ h_{i}^{-1}h_{j}\circ e_{z}$, while manifold $M$ has an atlas, enumerated by $h_i\in G$, $\lbrace(B(h_i(z),\rho),  e_{z}^{-1}\circ h_{i}^{-1})\rbrace_{i\in\N_{0}}$
with the same transition maps as $M_{\infty}$. 
Let $T_i\eqdef h_{i}\circ e_{z}\circ \varphi_{i}^{-1}: \;\varphi_{i}(\Omega_\rho)\to M$,  $i\in\N_0$, and note that this defines a smooth map $T:M_{\infty}^{(y_{k;i})}\to M$, since the values of $T_i$ are consistent on 
intersections of sets $\varphi_{i}(\Omega_\rho)$:
\begin{align}
h_{i}\circ e_{z}\circ \varphi_{i}^{-1}\circ [h_{j}\circ e_{z}\circ \varphi_{j}^{-1}]^{-1}= h_{i}\circ e_{z}\circ \psi_{ij} \circ [h_{j}\circ e_{z}]^{-1}=\\
h_{i}\circ e_{z}\circ e_{z}^{-1}\circ h_{i}^{-1}h_{j}\circ e_{z}\circ e_{z}^{-1}\circ h_j^{-1}
=\mathrm{id}.
\end{align}
Furthermore, $T$ is a diffeomorphism with $T^{-1}=\varphi_{i}\circ e_{z}^{-1}\circ h_{i}^{-1}$, consistently defined
on $B(g_{i}z,\rho)$, $i\in\N_{0}$. Note that (\ref{eq:metric_1.2a})
on $M_{\infty}^{(y_{k;i})}$ holds because it holds on $M$ with the same transition
map for every $k$, so the Riemannian metric on $M_{\infty}^{(y_{k;i})}$ in the normal
coordinates coincides with the Riemannian metric on $M$. In what
follows we will identify $M_{\infty}^{(y_{k;i})}$ as $M$.

2. Let now $(u_{k})$ be a bounded sequence in $H^{1,2}(M)$ and note
that its local profile associated with the sequence $(g_{k}h_{i})_{k\in\N}$
is given by 
\[
w_{i}=\wlim u_{k}\circ(g_{k}h_{i})\circ e_{z},
\]
and the global profile is by definition $w=w_{i}\circ\varphi_{i}^{-1}=w_{i}\circ e_{z}^{-1}\circ h_{i}^{-1}=\wlim u_{k}\circ g_{k}$
, which coincides with the profile of $(u_{k})$ as defined in Theorem
\ref{thm:FT} in (\ref{eq:profile}) relative to the sequence $(g_{k}).$
Consider now the local concentration defined by the array  $\lbrace w_{i}\rbrace_{i\in\N_{0}}$ of local profiles:
\begin{align}
W_{k}= & \sum_{i\in\N_{0}}\chi_{g_{k}h_{i}z}w_{i}\circ e_{y_{k;i}}^{-1}=
\sum_{i\in\N_{0}}\chi_{g_{k}h_{i}z}w_{i}\circ e_{z}^{-1}\circ h_{i}^{-1}\circ g_{k}^{-1}= \nonumber
\\
= &\sum_{i\in\N_{0}}\chi_{g_{k}h_{i}z}w\circ g_{k}^{-1}=w\circ g_{k}^{-1},\nonumber
\end{align}
which completes the proof. 
\end{proof}

\section{Appendix}

\subsection{Manifolds of bounded geometry and covering lemma}

In this appendix we give some elementary properties of manifolds of bounded geometry. All needed definition can be found e.g. in Chavel's book \cite{Chavel}. Let $M$ be an $N$-dimensional  Riemannian manifold of bounded geometry with a metric tensor $g$.  Let $v_g$ denote the Riemannian measure on $M$ and let $L_2(M)$ be the corresponding space of square  integrable functions.  For $k$ integer, and $f:M\rightarrow \C$ we denote by $\nabla f$ the  covariant derivative of $u$,  and by $|\nabla u|$ the norm of $\nabla u$ defined by a local chart i.e. 
\[
|\nabla f|^2 = g^{ij}\partial_{i} u  \partial_{j}  u\, 
\] 
where $g^{ij}$ are the components of the inverse matrix of the metric matrix $g=(g_{ij})$. 
The Sobolev space $H^{1,2}(M)$ is a completion of $C^{\infty}_o(M)$ with respect to the norm given by 
\[
\|f\|^2_{H^{1,2}}=  \|\nabla f\|^2_2  + \| f\|^2_2.
\]

We start with the following lemma, and refer to \cite{Eich} for the proof.

\begin{lem}\label{eichhorn}
Let $M$ be a Riemannian manifold of bounded geometry and let $0< r<r(M)$. If $k\in \N$ then there exists a constant $C_k$ dependent on the curvature bounds and $r$ but independent of $x\in M$, which bounds the $C^k$-norm of components $g_{ij}$ of the metric tensor $g$ and it inverse $g^{ij}$ in any normal coordinate system of radius not exceeding $r$ at any point $x\in M$.
\end{lem}

For any two points $x\in M$ and $0<r<r(M)$ let  
\[
e_{x}:\Omega_{r}\rightarrow B(x,r)
\]
denote a normal coordinate system at $x$ defined on the euclidean ball $\Omega_r$ centered at origin.

The boundedness of the derivatives of the Riemann curvature tensor is equivalent to the following lemma, cf.\cite{Shubin},
\begin{lem}
\label{lem:bdd-derivatives}
If the manifold $M$ has bounded geometry and  $0<r<r(M)$  then for any $\alpha\in\N_{0}^{N}$ there exists a constant $C_{\alpha}>0$,
such that $ $
\[
|D^\alpha e_{y}^{-1}\circ e_{x}(\xi)|\le C_{\alpha}\mbox{ whenever }x,y\in M,\text{and}\; B(x,r)\cap B(y,r)\not= \emptyset.
\]
\end{lem}

The next two statements  can be found  is many places  in  literature, cf. eg. \cite{GS}, \cite{Shubin}, \cite{LS}.  

\begin{lem}
\label{lem:covering}
Let $M$ be a $N$-dimensional connected Riemannian
mani\-fold with bounded geometry. Let $\rho>0$. There exists an at
most countable set $Y\in M$ such that 
\begin{align}
d(y,y') & \ge\rho/2\quad  \text{ whenever }\;y\neq y',\; y,y'\in Y,\label{eq:discrete}\\
  M  & = \bigcup_{y\in Y} B(y,r)\quad \text{ for any } \quad r>\rho .
\end{align}
Moreover   for any $r>\rho$ the multiplicity of the covering $\lbrace B(y,r)\rbrace_{y\in Y}$ 
is uniformly finite.
\end{lem}

\begin{lem}\label{lem:PU}
Let $M$, $Y$, $\rho$ and $r$ be as in Lemma~\ref{lem:covering}. There exists a smooth partition of unity $\{\chi_y\}_{y\in Y}$ on $M$, subordinated to the covering $\{B(y,\rho)\}_{y\in Y}$, such that  for any $\alpha\in\N_{0}^{N}$ there exists a constant $C_{\alpha}>0$,
such that 
\begin{equation}
|D^\alpha\chi_{y}|\le C_{\alpha} \label{eq:dalpha}
\end{equation}
for all $y\in Y$.
\end{lem}
 The following corollary  is the immediate consequence of Lemma~\ref{eichhorn} above.
\begin{cor}
\label{lem:2.1}Let $p\in(0,\infty)$ and $r\in(0,r(M))$. There exists
a constant $C>1$ such that for any $x\in M$ 
\begin{equation}
C^{-1}\int_{B(x,r)} |u|^{p}d\mu\le \int_{\Omega_{r}}|u\circ e_{x}\text{\ensuremath{|^{p}} dx \ensuremath{\le} C \ensuremath{\int}}_{B(x,r)}|u|^{p}d\mu,\label{eq:peq}
\end{equation}
and 
\[
C^{-1}\int_{{B(x,r)}}|\nabla u|^{2}d\mu \le \int_{\Omega_{r}} \sum_{i=1}^N|\partial_i(u\ensuremath{\circ e_{x}})|^{2}dx
 \ensuremath{\le}C\ensuremath{\int}_{B(x,r)}|\nabla u|^{2}d\mu
\]
\end{cor}
 
 We finish this subsection by recalling a technical but useful equivalent norm in    $H^{1,2}(M)$, cf. \cite{GS} or \cite[Chapter 7]{Triebel92},
\begin{lem}\label{eqv:norms}
Let $\{B(y_i,r)\}$ be a locally uniformly finite covering of  N-dimensional manifold $M$ with bounded geometry, $r\in(0,r(M))$ and let $\{\chi_i\}$ be a partition of unity subordinated to the covering $\{B(y_i,r)\}$ as in Lemma~\ref{lem:PU}. Then    
\begin{equation}\label{eqv:norm1}
\nl f\nr_{H^{1,2}(M)} = \Big( \sum_i \|\chi_i f\circ \exp_{y_i}\|^2_{H^{1,2}(\R^N)} \Big)^{1/2}  
\end{equation}
is an equivalent norm in  $H^{1,2}(M)$. Moreover   
\[
\|f\|_{H^{1,2}(M)} \sim \nl f\nr_{H^{1,2}(M)} \sim  \Big( \sum_i \|\chi_i f\|^2_{H^{1,2}(M)} \Big)^{1/2}.
\]
\end{lem}


\subsection{Gluing manifolds}
We use a particular case of gluing theorem in Gallier et al, \cite[Theorem 3.1]{Gallier3} 
\begin{defn}
\label{def: gluing data}(\cite[Definition 3.1]{Gallier3}, \cite[Definition 8.1]{Gallier}).
A \emph{set of gluing data} is a triple $(\lbrace\Omega_{i}\rbrace_{i\in\N_{0}},\lbrace\Omega_{ij}\rbrace_{i,j\in\N_{0}},\lbrace\psi_{ji}\rbrace_{(i,j)\in\mathbb{K}})$
satisfying the following properties:

(1) For every $i\in\N_{0}$, the set $\Omega_{i}$ is a nonempty open
subset of $\R^{N}$ and the sets $\lbrace\Omega_{i}\rbrace_{i\in\N_{0}}$
are pairwise disjoint;

(2) For every pair $i,j\in\N_{0}$, the set $\Omega_{ij}$ is an open
subset of $\Omega_{i}$. Furthermore, $\Omega_{ii}=\Omega_{i}$ and
$\Omega_{ji}\neq\emptyset$ if and only if $\Omega_{ij}\neq\emptyset$; 

(3) $\mathbb{K}=\lbrace(i,j)\in\N_{0}\times\N_{0}:\Omega_{ij}\neq\emptyset\rbrace$,
$\psi_{ji}:\Omega_{ij}\to\Omega_{ji}$ is a diffeomorphism for every
$(i,j)\in\mathbb{K}$, and the following conditions hold:

(a) $\psi_{ii}=\mathrm{id}|_{\Omega_{i}}$, for all $i\in\N_{0}$, 

(b) $\psi_{ij}=\psi_{ji}^{-1}$, for all $(i,j)\in\mathbb{K}$,

(c) For all $i,j,k\in\N_{0}$, if $\Omega_{ji}\cap\Omega_{jk}\neq\emptyset$,
then $\psi_{ij}(\Omega_{ji}\cap\Omega_{jk})=\Omega_{ij}\cap\Omega_{ik}$,
and $\psi_{ki}(x)=\psi_{kj}\circ\psi_{ji}(x)$, for all $x\in\Omega_{ij}\cap\Omega_{ik}$;

(4) For every pair$(i,j)\in\mathbb{K}$, with $i\neq j$, for every
$x\in\partial\Omega_{ij}\cap\Omega_{i}$ and every $y\in\partial\Omega_{ji}\cap\Omega_{j}$,
there are open balls $V_{x}$ and $V_{y}$ centered at $x$ and $y$
so that no point of $V_{y}\cap\Omega_{ji}$ is the image of any point
of $V_{x}\cap\Omega_{ij}$ by $\psi_{ji}$.

Each set $\Omega_{i}$ is called \emph{parametrization domain} or
\emph{$p$-domain}, each nonempty set $\Omega_{ij}$ is called a \emph{gluing
domain}, and each map $\psi_{ij}$ is called \emph{transition map}
or \emph{gluing map}. 
\end{defn}
\begin{thm}
\label{thm:gluing}(\cite[Theorem 3.1]{Gallier3} )For every set of
gluing data, \[(\lbrace\Omega_{i}\rbrace_{i\in\N_{0}},\lbrace\Omega_{ij}\rbrace_{i,j\in\N_{0}},\lbrace\psi_{ji}\rbrace_{(i,j)\in\mathbb{K}}),\]
there exists a $N$-dimensional smooth manifold $M$ 
an atlas $(U_i,\tau_i)_i$ of $M$ such that $\tau_i(U_i)=\Omega_i$, whose transition maps are $\tau_j\circ\tau_i^{-1}= \psi_{ji}:\Omega_{ij}\to\Omega_{ji}$. 
$i,j\in\N_{0}$.
\end{thm}
\begin{rem}
Note that the theorem 
does not provide any specifics about the maps $\tau_{i}$ which are obviously not uniquely defined.
\end{rem}
\begin{cor}
\label{cor:gluing} 
 Let  $0<\rho<r<a$ and  let $\Omega_\rho\subset \Omega_r\subset \Omega_a$ be  balls in $\R^N$ centered at the origin  with  radius $\rho$, $r$ and $a$ respectively.  Let $\lbrace \widetilde{\psi}_{ij}\rbrace_{i,j\in\N_{0}}$ be a family of smooth open maps
$\widetilde{\psi}_{ij}:\Omega_{r}\to\Omega_{a}$.   Assume that a family $\lbrace \psi_{ji}=\widetilde{\psi}_{ji}|_{\Omega_\rho}\rbrace_{i,j\in\N_{0}}$ satisfies the following conditions:

(i) $\psi_{ii}=\mathrm{id}$, $i\in\N_{0}$;

(ii) $\psi_{ji}$ is a diffeomorphism between $\Omega_{ij}\eqdef\psi_{ij}(\Omega_{\rho})\cap\Omega_{\text{\ensuremath{\rho}}}$
and $\Omega_{ji}$, $i,j\in\N_{0}$, whenever $\Omega_{ji}\neq\emptyset$;

(iii) $\psi_{ij}=\psi_{ji}^{-1}$ on $\Omega_{ji}$, whenever $\Omega_{ji}\neq\emptyset$,
$i,j\in\N_{0}$;

(iv) $\psi_{ij}(\Omega_{ji}\cap\Omega_{jk})=\Omega_{ij}\cap\Omega_{ik}$,
and $\psi_{ki}(x)=\psi_{kj}\circ\psi_{ji}(x)$ for all \textup{$x\in\Omega_{ij}\cap\Omega_{ik}$},
$i,j,k\in\N_{0}$;

(v) for all $(i,j)\in \mathbb{K}\eqdef\lbrace(i,j)\in\N_{0}\times\N_{0}:\Omega_{ij}\neq\emptyset\rbrace$ and all $x\in \partial\Omega_{ij}\cap \Omega_\rho $ $\psi_{ji}(x)\in \partial\Omega_{ji}\cap \partial\Omega_\rho $.

Then there exists a smooth differential manifold $M$ 
with an atlas $\lbrace (U_i,\tau_{i})\rbrace_{i\in\N_{0}}$, such that $\tau_i(U_i)= \Omega_\rho$ for any $i\in \N_0$ and 
whose transition maps $\tau_{j}\circ\tau_{i}^{-1}$ are $\psi_{ji}:\Omega_{ij}\to\Omega_{ji}$.
$i,j\in\N_{0}$. 
\end{cor}

\begin{proof}
Fix an enumeration $(z_{i})_{i\in\N_{0}}$ of the lattice $3a\Z^{N}\subset\R^{N}$.
Set $\Omega_{i}'\eqdef z_{i}+\Omega_{\rho}$, $i\in\N_{0}$, and $\Omega_{ij}'\eqdef\Omega_{ij}+z_{i}$,
$\psi_{ij}'\eqdef\psi_{ij}(\cdot-z_{j})+z_{i}$, for $(i,j)\in\mathbb{K}$. 
The corollary is immediate from Theorem \ref{thm:gluing} once we
show that $(\lbrace\Omega'_{i}\rbrace_{i\in\N_{0}},\lbrace\Omega'_{ij}\rbrace_{i,j\in\N_{0}},\lbrace\psi_{ij}'\rbrace_{(i,j)\in\mathbb{K}})$
is a set of gluing data according to Definition \ref{def: gluing data}.
Conditions of the definition verify as follows.

\emph{Condition (1)} is immediate since $3a>2\rho$. %

\emph{Condition (2).} The sets $\Omega_{ij}$ (and thus $\Omega_{ij}'$)
are open since the maps $\psi_{ji}$ are open. The relation $\Omega_{ij}'\subset\Omega_{i}'$
follows from $\Omega_{ij}\subset\Omega_{\rho}$ in \emph{(ii)}. By
\emph{(i)} we have $\Omega_{ii}=\Omega_{\rho}$ and thus $\Omega_{ii}'=\Omega_{i}'$.
If $\Omega'_{ij}\neq\emptyset$, then $\Omega_{ij}\neq\emptyset$,
and since $\psi_{ij}$ is the inverse of $\psi_{ji}$, $\Omega_{ji}\eqdef\psi_{ji}(\Omega_{\rho}\cap\psi_{ij}\Omega_{\text{\ensuremath{\rho}}})=\psi_{ji}\Omega_{ij}\neq\emptyset$
. Thus $\Omega'_{ji}\neq\emptyset$.

\emph{Conditions (3): properties (a)}, \emph{(b)}, and \emph{(c) are
immediate, respectively, from (i)}, \emph{(iii)}, and \emph{(iv)}.

\emph{Condition (4)}. 
Let  $x\in \partial\Omega'_{ij}\cap \Omega_\rho(z_i)$ and $y\in \partial \Omega'_{ji}\cap\Omega_\rho(z_j)$.   Then $\tilde{x}=x - z_i \in \partial\Omega_{ij}\cap \Omega_\rho$ and $\tilde{x}=x - z_j \in \partial \Omega_{ji}\cap\Omega_\rho(z_j)$. By assumption \emph{(v)} we have $\tilde{y}\not= \psi_{ji}(\tilde{x})$. In consequence there exist Euclidean balls  $\Omega(\tilde{x},\varepsilon)$ and $\Omega(\tilde{y},\varepsilon)$ such that no point of $\Omega(\tilde{y},\varepsilon)\cap \Omega_\rho$ is an image of $\Omega(\tilde{x},\varepsilon)\cap \Omega_\rho$. 
\end{proof}


\end{document}